\documentclass{CSML}

\def\dOi{11(4:20)2015}
\lmcsheading%
{\dOi}
{1--44}
{}
{}
{Dec.~\phantom08, 2013}
{Dec.~29, 2015}
{}

\ACMCCS{[{\bf Mathematics of computing}]: Mathematical analysis; Continuous mathematics; [{\bf Theory of computation}]: Logic---Constructive mathematics}

\usepackage[utf8]{inputenc}
\usepackage[english]{babel}
\usepackage{hyperref}
\usepackage{amsmath, amssymb, amstext, amsthm}
\usepackage{amscd}
\usepackage{graphicx}
\usepackage{calrsfs}

\theoremstyle{plain}
\newtheorem{theorem}{Theorem}[section]
\newtheorem{proposition}[theorem]{Proposition}
\newtheorem{corollary}[theorem]{Corollary}
\newtheorem{lemma}[theorem]{Lemma}
\newtheorem{conjecture}[theorem]{Conjecture}
\newtheorem{fct}[theorem]{Fact}

\theoremstyle{definition}
\newtheorem{Definition}[theorem]{Definition}
\theoremstyle{remark}
\newtheorem{Remark}[theorem]{Remark}
\newtheorem{caveat}[theorem]{Caveat}

\newcommand{\demph}[1]{\emph{#1}}
\newcommand{\product}[2]{(#1,#2)}
\newcommand{\tuple}[2]{\langle #1,#2\rangle}
\newcommand{\Fix}{\operatorname{Fix}}

\newcommand{\R}{\mathbb{R}}
\newcommand{\N}{\mathbb{N}}
\newcommand{\Q}{\mathbb{Q}}
\newcommand{\comp}[1]{{#1}^C}
\newcommand{\intr}[1]{{#1}^\circ}
\newcommand{\norm}[1]{||#1||}
\newcommand{\bignorm}[1]{\Big|\Big|#1\Big|\Big|}

\newcommand{\clos}[1]{\overline{#1}}
\newcommand{\spann}{\operatorname{span}}
\renewcommand{\epsilon}{\varepsilon}
\newcommand{\dom}{\operatorname{dom}}
\newcommand{\Co}{\operatorname{\mathcal{C}}}
\newcommand{\RCo}[2]{[#1\rightarrow #2]}

\renewcommand{\d}{\operatorname{d}}

\newcommand{\K}{\mathcal{K}}

\newcommand{\Field}{\mathbb{F}}
\newcommand{\C}{\operatorname{C}}
\newcommand{\Ne}{\operatorname{\mathcal{N}}}
\newcommand{\id}{\operatorname{id}}
\newcommand{\LLPO}{\operatorname{LLPO}}
\newcommand{\XC}{\operatorname{ConvC}}
\renewcommand{\O}{\mathcal{O}}

\newcommand{\dist}{\operatorname{dist}}

\newcommand{\HB}{\operatorname{HB}}
\newcommand{\A}{\mathcal{A}}
\newcommand{\UC}{\operatorname{UC}}

\renewcommand{\conv}{\operatorname{co}}

\renewcommand{\S}{\mathbb{S}}

\newcommand{\V}{\mathcal{V}}

\newcommand{\BGK}{\operatorname{BGK}}

\newcommand{\BFT}{\operatorname{BFT}}

\newcommand{\IVT}{\operatorname{IVT}}

\newcommand{\Conv}{\operatorname{Conv}}

\newcommand{\WKL}{\operatorname{WKL}}

\newcommand{\im}{\operatorname{im}}

\newcommand{\WBGK}{\operatorname{WBGK}}
\newcommand{\CC}{\operatorname{CC}}
\newcommand{\w}{\operatorname{w}}
\newcommand{\n}{\operatorname{n}}

\newcommand{\Proj}{\operatorname{Proj}}
\newcommand{\re}{\operatorname{enum}}

\makeindex
\begin{document}

\author[E.~Neumann]{Eike Neumann}	
\address{Technische Universit\"at Darmstadt, Germany}
\email{eike.neumann@stud.tu-darmstadt.de}
\thanks{The author was partly supported by the German
   Research Foundation (DFG) with Project Zi 1009/4-1 and by the Royal Society International Exchange Grant IE111233}
\titlecomment{{\lsuper*}This paper is essentially a condensed version of the author's master's thesis \cite{thesis}, written under the supervision of Ulrich Kohlenbach at Technische Universit\"at Darmstadt.}

\title[Computational Problems in Metric Fixed Point Theory]{Computational Problems in Metric Fixed Point Theory and their Weihrauch Degrees\rsuper*}
\begin{abstract}
We study the computational difficulty of the problem of finding fixed points of nonexpansive mappings in uniformly convex Banach spaces. We show that the fixed point sets of computable nonexpansive self-maps of a nonempty, computably weakly closed, convex and bounded subset of a computable real Hilbert space are precisely the nonempty, co-r.e.\ weakly closed, convex subsets of the domain. A uniform version of this result allows us to determine the Weihrauch degree of the Browder-G\"ohde-Kirk theorem in computable real Hilbert space: it is equivalent to a closed choice principle, which receives as input a closed, convex and bounded set via negative information in the weak topology and outputs a point in the set, represented in the strong topology. While in finite dimensional uniformly convex Banach spaces, computable nonexpansive mappings always have computable fixed points, on the unit ball in infinite-dimensional separable Hilbert space the Browder-G\"ohde-Kirk theorem becomes Weihrauch-equivalent to the limit operator, and on the Hilbert cube it is equivalent to Weak K\H{o}nig's Lemma. In particular, computable nonexpansive mappings may not have any computable fixed points in infinite dimension. We also study the computational difficulty of the problem of finding rates of convergence for a large class of fixed point iterations, which generalise both Halpern- and Mann-iterations, and prove that the problem of finding rates of convergence already on the unit interval is equivalent to the limit operator. 
\end{abstract}
\maketitle

\section{Introduction}
Metric fixed point theory is the study of fixed point properties of mappings that arise from the geometric structure of the underlying space or the geometric properties of the mappings themselves. An important classical framework for metric fixed point theory is the study of nonexpansive mappings in uniformly convex Banach spaces. A Banach space $E$ is called \emph{strictly convex}, if for all $x, y \in E$ with $x \neq y$ and $\norm{x} = \norm{y} = 1$, we have $\norm{\tfrac{x + y}{2}} < 1$.
It is called \emph{uniformly convex}, if
\[\forall \varepsilon \in(0,2]. \exists \delta \in (0,1].\forall x,y \in B_E.\left(\norm{x - y} \geq \varepsilon \rightarrow \bignorm{\frac{x + y}{2}} \leq 1 - \delta \right) .\]
Here, $B_E$ denotes the closed unit ball of $E$. Clearly, every uniformly convex Banach space is strictly convex. A function $\eta_E\colon(0,2]\to(0,1]$ witnessing the existential quantifier is called a \emph{modulus of convexity} for $E$. By the parallelogram law, every Hilbert space $H$ is uniformly convex with computable modulus of convexity $\eta_H(\varepsilon) = 1 - \sqrt{1 - \frac{\varepsilon^2}{4}}$. More generally, all $L^p$-spaces with $1 < p < \infty$ are uniformly convex with a computable modulus of uniform convexity. A mapping $f\colon\subseteq E \to E$ is called \emph{nonexpansive} if it is Lipschitz-continuous with Lipschitz-constant one, i.e. if
\[\norm{f(x) - f(y)} \leq \norm{x - y} \;\text{ for all } x,y\in \dom f.\]
We have the following existence result:
\begin{theorem}[Browder-Göhde-Kirk]\label{thm: Browder-Goehde-Kirk Theorem}
Let $E$ be a uniformly convex Banach space, let ${K \subseteq E}$ be nonempty, bounded, closed, and convex, and let $f\colon K \to K$ be nonexpansive. Then $f$ has a fixed point.\qed
\end{theorem}
Theorem \ref{thm: Browder-Goehde-Kirk Theorem} was proved independently by Browder \cite{BrowderFixedPointB}, Göhde \cite{Goehde}, and Kirk \cite{KirkFixedPoint} in 1965 (Kirk's version is even more general than the version stated here). Throughout this paper we denote the fixed point set of a mapping $f$ by $\Fix(f)$. A considerable amount of attention is dedicated to the study of so-called fixed point iterations, which start with an initial guess $x_0$ for a fixed point of $f$ and successively improve the guess by applying a computable operation, which yields a sequence $(x_n)_n$ of points in $K$ that is then shown to converge (weakly or strongly) to a fixed point. Many of these results are modifications of either of two classical theorems.
\begin{theorem}[Wittmann, \cite{Wittmann}]\label{thm: Halpern's theorem}
Let $H$ be a Hilbert space, let $K \subseteq H$ be nonempty, bounded, closed, and convex, and let $f\colon K \to K$ be nonexpansive. Choose a starting point $x \in K$ and an ``anchor point'' $y \in K$ and consider the sequence $(x_n)_n$, where $x_0 = x$ and $x_{n+1} = \tfrac{1}{n + 2} y + (1 - \tfrac{1}{n + 2})f(x_n)$. Then the sequence $(x_n)_n$ converges to the uniquely defined fixed point of $f$ which is closest to the anchor point $y$.\qed
\end{theorem}

\begin{theorem}[Krasnoselski, \cite{Krasnoselski}]\label{thm: Krasnoselski-Mann theorem}
Let $E$ be a uniformly convex Banach space, let ${K \subseteq E}$ be nonempty, bounded, closed, and convex, and let $f\colon K \to K$ be nonexpansive and $f(K)$ be compact. Then for any $x \in K$ the sequence $(x_n)_n$, where $x_0 = x$ and $x_{n+1} = (f(x_n) + x_n)/2$, converges to a fixed point of $f$.\qed
\end{theorem}

The iteration employed in Theorem \ref{thm: Halpern's theorem} is a special case of a general iteration scheme, typically referred to as \emph{Halpern iteration}, as it was first introduced by Halpern \cite{Halpern}. It has the general form $x_{n+1} = (1 - \alpha_n) y + \alpha_nf(x_n)$, where $\alpha_n \in (0,1)$, and the iteration can be shown to converge if certain conditions are imposed on $(\alpha_n)_n$. The iteration used in Theorem \ref{thm: Krasnoselski-Mann theorem} can be similarly generalised to the scheme $x_{n+1} = (1 - \alpha_n)x_n + \alpha_nf(x_n)$, and again there are certain conditions that guarantee convergence. This iteration scheme is typically called \emph{Krasnoselski-Mann iteration} or simply \emph{Mann iteration}. In Hilbert space, Krasnoselski's iteration converges weakly to a fixed point, even in the absence of compactness (cf.~\cite{KrasnoselskiWeakConvergence}). While these iterations do allow us to compute a sequence of approximations which is guaranteed to eventually converge to a fixed point, it is well known that the requirement of mere convergence is too weak to constitute a satisfactory notion of effective approximation, as there exist for instance computable sequences of rational numbers whose limit encodes the special halting problem (cf.~\cite{SpeckerSequence}). It is hence important to understand the quantitative convergence behaviour of the approximation sequence. Quantitative aspects of metric fixed point theory have been very successfully studied within the programme of \emph{proof mining} (the standard reference is \cite{Kohlenbook}, see also e.g.~\cite{Kohlenbach, KohlenbachMetric, KohlenbachBRS, Leustean, KohlenbachLeustean, KoernleinKohlenbach,SchadeKohlenbach}), which is concerned with the extraction of hidden effective data from non-effective proofs. Most of the applications of proof mining in fixed point theory focus on the extraction of either of two types of effective data. Firstly, one considers \emph{rates of asymptotic regularity} of the iteration, which in this context mean rates of convergence of the sequence $(\norm{f(x_n) - x_n})_n$ towards zero. These allow us to compute arbitrarily good $\varepsilon$-fixed points, i.e. points $x_{\varepsilon}$ satisfying ${\norm{f(x_{\varepsilon}) - x_{\varepsilon}} < \varepsilon}$, up to arbitrary precision with an a-priori running time estimate. Secondly, one considers so-called \emph{rates of metastability} (see also \cite{KreiselNoCounterexample1,KreiselNoCounterexample2,TaoMetastability1,TaoMetastability2}), which constitute a more refined quantitative measure of approximation quality. A function $\Phi\colon\N^{\N}\times\N\to\N$ is called a rate of metastability for the sequence $(x_n)_n$ if it satisfies
\begin{equation}\label{eq: metastability}\forall n\in \N.\forall g\colon\N\to\N.\exists k \leq \Phi(g,n).\forall i,j\in[k;k + g(k)]\left(\norm{x_{i} - x_j} < 2^{-n} \right).
\end{equation}
Note that \eqref{eq: metastability} is classically (but not constructively) equivalent to the statement that $(x_n)_n$ is a Cauchy sequence, so that metastability can be viewed as a finitary version of convergence. Also note that in the case of Krasnoselski's iteration, asymptotic regularity is the special case of metastability where $g(k) = 1$ for all $k\in\N$. Of course, both types of information are strictly weaker than actual rates of convergence. In fact, effective uniform rates of convergence cannot exist, as the existence result fails to be computably realisable already in the case where $K = [0,1]$.
\begin{theorem}[\cite{Kohlenbach}]\label{uniform uncomputability of fixed points}
The multi-valued operator which receives as input a nonexpansive self map $f$ of the compact unit interval $[0,1]$ and returns some fixed point of $f$ is not computable.\qed
\end{theorem}
While Theorem \ref{uniform uncomputability of fixed points} already shows that there exists no algorithm for computing a rate of convergence for Krasnoselski's or Halpern's iteration uniformly in the input function and the starting point, it leaves several questions open: whether every computable nonexpansive mapping has a computable fixed point, whether there exist non-uniformly computable rates of convergence for the Mann- or Halpern-iteration for every computable nonexpansive mapping, at least for certain suitable starting points, whether fixed points are uniformly computable relative to discrete advice, what the exact relation between the computational content of the three theorems is, and how their computational content relates to the computational content of other mathematical theorems, such as Brouwer's fixed point theorem. 

In this paper, we study the computational content of Theorems \ref{thm: Browder-Goehde-Kirk Theorem}, \ref{thm: Krasnoselski-Mann theorem}, and \ref{thm: Halpern's theorem}, as well as related computational problems in terms of \emph{Weihrauch degrees}, which have been proposed by Gherardi and Brattka \cite{WeihrauchDegrees} as a framework for classifying mathematical theorems according to their computational content. Many classical mathematical theorems have been classified over the recent years. Recently, Brattka, Le Roux and Pauly \cite{ConnectedChoice} have shown that Brouwer's fixed point theorem in dimension $n$ is equivalent to the closed choice principle on the closed unit ball in $\R^n$ restricted to connected sets, and that it is equivalent to Weak K\H{o}nig's Lemma from dimension three upwards. Their work is based on a characterisation of the fixed point sets of computable self-maps of the unit ball in $\R^n$, due to Miller \cite{Miller}. We provide a similar characterisation for the fixed point sets of computable nonexpansive self-maps of nonempty, convex, closed, and bounded subsets of computable Hilbert space, which we can use to determine the Weihrauch degree of the Browder-G\"ohde-Kirk theorem. This will in particular allow us to compare the computational content of the Browder-Göhde-Kirk theorem and the problem of finding rates of convergence for fixed point iterations to Brouwer's classic result. 
\section{Preliminaries}
Here we review some basic notions from computable and functional analysis and the theory of Weihrauch reducibility. Most of the results in this section are more or less folklore, and none of them are original, except maybe Proposition \ref{prop: equivalence lower semi-located and computably closed implies locally compact}. Standard references in computable analysis are \cite{PourElRichards} and \cite{Weih}. A more general treatment of the theory of computable metric spaces can be found in \cite{BrattkaGherardi} and
\cite{BrattkaPresser}. The results in functional analysis reviewed here can for instance be found in \cite{Megginson} or \cite{Werner}. We will closely follow the approach to computable analysis taken by Matthias Schr\"oder \cite{SchroederPhD}, and more recently by Arno Pauly \cite{PaulyRepresented}, particularly concerning the canonical constructions of hyperspaces. Also, we adopt most of the notation and terminology from \cite{PaulyRepresented}, which differs from standard terminology at certain points (see Caveat \ref{caveat: terminology}). 

A \emph{numbering} of a nonempty countable set $S$ is a surjective partial mapping ${\nu\colon\subseteq \N \to S}$. A \emph{representation} of a nonempty set $X$ is a surjective partial mapping ${\delta\colon\subseteq \N^\N \to X}$. If $\delta$ is a representation of $X$, we call the tuple $(X,\delta)$ a \emph{represented space}. If the underlying representation is clear from context, we will often simply write $X$ for $(X,\delta)$ and by convention denote the underlying representation $\delta$ of $X$ by $\delta_X$. If $\delta$ and $\varepsilon$ are representations of the same set $X$, we denote \emph{continuous reduction} by $\delta \leq_t \varepsilon$ and \emph{computable reduction} by $\delta \leq \varepsilon$. A representation $\delta\colon\subseteq\N^\N\to X$ is \emph{admissible} if it is continuous and maximal with respect to continuous reduction. A \emph{represented topological space}\footnote{Like Schr\"oder, and as opposed to Pauly, we will mostly work with a specific topology for $X$ in mind. To emphasize this, we call the spaces of interest represented \emph{topological} spaces, rather than ``represented spaces''. In general, the topology on $X$ will \emph{not} be the final topology of its representation, and topological continuity may differ from realiser-continuity if the topology of $X$ is not sequential.} is a tuple $(X,\delta)$, where $X$ is a topological space and $\delta$ is an admissible representation for $X$. If the representation is clear from the context, we will simply write $X$ for $(X,\delta)$ and by convention denote the underlying admissible representation $\delta$ by $\delta_X$. We say that a partial mapping $F\colon\subseteq \N^\N \to \N^\N$ is a \emph{realiser} for a partial multi-valued mapping (or ``multimapping'') $f\colon\subseteq X \rightrightarrows Y$ between represented spaces and write $F\vdash f$ if $\delta_Y(F(p)) \in f(x)$, whenever $\delta_X(p) = x$. We call $f$ \emph{computable} if it has a computable realiser, and \emph{realiser-continuous} if it has a continuous realiser. If we want to emphasise the underlying representations, we will write that $f$ is $(\delta_X,\delta_Y)$-computable or $(\delta_X,\delta_Y)$-continuous respectively. We denote by $\rho$ the standard representation of real numbers. If $\delta\colon \subseteq \N^\N\to X$ and $\varepsilon\colon\subseteq \N^\N \to Y$ are representations (or numberings), we denote by $[\delta\to\varepsilon]$ the canonical representation of the space $\RCo{X}{Y}$ of functions with continuous realiser. If $X$ and $Y$ are represented topological spaces, then $\RCo{X}{Y}$ coincides with the space $\Co_{\operatorname{seq}}(X,Y)$ of sequentially continuous functions from $X$ to $Y$. If $X$ is first-countable, $\RCo{X}{Y}$ furthermore coincides with the space $\Co(X,Y)$ of continuous functions from $X$ to $Y$. Moreover, we denote by $\delta\times\varepsilon$ the canonical representation of the product space $X\times Y$, by $\delta^\omega$ the canonical representation of the space $X^\N$ and by $\delta^*$ the canonical representation of ${X^* = \bigcup_{n \in \N} X^n}$. If $\delta$ and $\varepsilon$ are representations of the same space $X$, we let $\delta\sqcap\varepsilon = \pi_0(\delta\times\varepsilon)|^{\Delta(X)}$, where $\Delta(X) = \{(x,y) \in X^2\;|\; x = y\}$ and $\pi_0$ is the projection onto the first coordinate. If $K \subset X$ is a subset of $X$, we sometimes write $\delta_K$ for $\delta_X|^K$.

As already mentioned, the following constructions are essentially due to \cite{SchroederPhD} and \cite{PaulyRepresented}. For any represented space $X$, the canonical function-space construction gives rise to canonical representations of the hyperspaces of ``open'' and ``closed'' subsets of $X$, by postulating that openness corresponds to semi-decidability. Let $\S = \{0,1\}$ denote Sierpi\'nski space with topology $\{\emptyset, \{0,1\}, \{1\} \}$ and representation 
\[\sigma(p) = 0 \;:\Leftrightarrow\; p = 0.\]
The characteristic function $\chi_U\colon X \to \{0,1\}$ of a set $U$ is defined as $\chi_U(x) = 1 :\Leftrightarrow x \in U$. 
\begin{Definition}\label{Def: canonical representation of open and closed sets}
Let $X$ be a represented space. We call a set $U \subseteq X$ \emph{open}, if its characteristic function $\chi_U\colon X \to \S$ is realiser-continuous, i.e. $\chi_U \in \RCo{X}{\S}$. A $\theta^X$-name of an open set $U\subseteq X$ is a $[\delta_X\to\sigma]$-name of its characteristic function $\chi_U$. The set of all open subsets of $X$ with representation $\theta^X$ defines the represented space $\O(X)$. Dually, we define the represented space $\A(X)$ of closed subsets of $X$ by identifying a closed set $A \subseteq X$ with its complement in $\O(X)$ and call the underlying representation $\psi^X$.
\end{Definition}
We will often just write $\psi$ for $\psi^X$ if the underlying space is clear from context. We call the computable points of $\O(X)$ \emph{semi-decidable} and the computable points in $\A(X)$ \emph{co-semi-decidable}. Note that, just like the notion of realiser continuity may differ from topological continuity, the notions of closedness and openness for subsets of represented spaces are a-priori different from the notions of topological openness and closedness. If $X$ is an admissibly represented topological space, then the set $\O(X)$ coincides with the set of all sequentially open subsets of $X$ and $\A(X)$ coincides with the set of all sequentially closed subsets of $X$. If in addition $X$ is second-countable then $\O(X)$ is the set of open subsets of $X$ and $\A(X)$ is the set of all closed subsets of $X$.

\begin{caveat}\label{caveat: terminology}
Note that the terminology introduced here, which is mainly due to \cite{PaulyRepresented}, is different from the usual terminology used in computable analysis, which is for instance used in Weihrauch's book \cite{Weih}. In \cite{Weih}, the space $\A(X)$ is denoted by $\A_{>}(X)$ and its computable elements are called co-r.e.\ closed, rather than co-semi-decidable. Although we have introduced our $\A(X)$ as ``the space of closed subsets'' of $X$, we deliberately refrain from referring to its computable points as ``computably closed'', so as to avoid confusion with topological closedness on one hand, and with Weihrauch's terminology on the other. The symbol $\A(X)$ is used in \cite{Weih} to denote the space of closed and overt subsets of $X$, to be introduced below. Also note that in the abstract we used \emph{Weihrauch's} terminology.
\end{caveat}

The space $\A(X)$ can be thought of as ``the space of closed sets encoded via negative information''. The following definition provides in a certain sense a notion of ``closed sets encoded via \emph{positive} information''.

\begin{Definition}\label{Def: canonical representation of overt sets}
Let $X$ be a represented space. We define the \demph{represented space} $\V(X)$ \demph{of overt closed subsets of} $X$ to be the represented space of closed subsets of $X$, where a closed set $A \subseteq X$ is represented by a $[\theta^X \to \S]$-name of the function
\[\operatorname{intersects?}_{A}\colon \O(X) \to \S,\; U \mapsto \begin{cases}1 &\text{if }U\cap A \neq \emptyset\text{,}  \\ 0 &\text{otherwise.} \end{cases}\]
\end{Definition}

We denote the standard representation of $\V(X)$ by $\upsilon^X$ or simply $\upsilon$ and call the computable points of $\V(X)$ \demph{computably overt}. Computably overt closed sets are those, for which intersection with an open set can be effectively verified. The space $\V(X)$ hence corresponds to the space $\A_{<}(X)$ in \cite{Weih}, and hence is sometimes called the ``space of closed sets, represented with positive information''. In \cite{PaulyRepresented} it is argued that from an intrinsic perspective, the word ``closed'' is rather misleading, because the closure properties of the space $\V(X)$ differ significantly from the closure properties of closed sets (e.g. union is computable but intersection is not, the image of a computably overt set under a computable function is computably overt, but the preimage is not), and we agree with this position. Overtness is related to effective separability, which yields a convenient criterion for computable overtness (see e.g.~\cite[Theorem 3.8 (1)]{BrattkaPresser}).

\begin{proposition}\label{prop: overtness and recursive enumerability}
Let $X$ be a separable represented topological space. Define a representation $\delta_{\re}$ of the set of nonempty closed subsets of $X$ as follows:
\[\delta_{\re}(p) = A \; :\Leftrightarrow \delta_X^{\omega}(p) \text{ is dense in } A.\]
Then $\delta_{\re} \leq \upsilon\big|^{\V(X)\setminus\{\emptyset\}}$.
\end{proposition}
\begin{proof}
Suppose we are given a dense sequence $(x_n)_n$ in a closed set $A$, and an open set $U \in \O(X)$. In order to verify if $A\cap U$, check if there exists $n \in \N$ such that $x_n \in U$. This proves the claim.
\end{proof}

Our representation $\delta_{\re}$ is called $\delta_{\operatorname{range}}$ in \cite{BrattkaGherardi} and \cite{BrattkaPresser}. Next we define the canonical representation of the hyperspace of compact subsets of a Hausdorff represented topological space $X$. In a countably based $T_1$ space, a compact set can be represented as a list of all its finite open covers by basic neighbourhoods. It is easy to see that this representation is characterised by the property that containment in an open set is semi-decidable. This can be used to generalise the definition to arbitrary represented topological spaces, and in fact to arbitrary represented spaces. For the sake of simplicity we restrict ourselves to the case of Hausdorff represented topological spaces.
\begin{Definition}\label{Def: canonical representations of compact sets}
Let $X$ be a Hausdorff represented topological space. The represented space $\K(X)$ of compact subsets of $X$ is the set of all compact subsets of $X$, where a compact set $K \in \K(X)$ is represented as a $[\theta^X\to\sigma]$-name of the function $\operatorname{contained?}_K\colon \O(X) \to \S$, 
\[\operatorname{contained?}_K(U) = 1 \Leftrightarrow K \subseteq U.\]
\end{Definition}

We denote the canonical representation of $\K(X)$ by $\kappa$ and call the computable points of $\K(X)$ \emph{computably compact}. Note that, like $\A(X)$, our space $\K(X)$ only encodes ``negative'' information on compact sets. Weihrauch \cite{Weih} hence uses the notation ``$\kappa_{>}$'' for our $\kappa$. Similarly as in the case of $\A(X)$, computable points in our $\K(X)$ are called ``co-r.e.\ compact'' by some authors. Definition \ref{Def: canonical representations of compact sets} can be generalised to arbitrary represented spaces, essentially by using the same approach as in Definition \ref{Def: canonical representation of open and closed sets}, and \emph{calling} a subset $K$ of a represented space $X$ compact if the function $\operatorname{contained?}_K$ is an element of $\RCo{\O(X)}{\S}$. In general this will only yield a representation of the space of \emph{saturated} compact sets (cf.~\cite{PaulyRepresented}), or a multi-valued representation of the space of compact sets (cf.~\cite{SchroederPhD}).  If $X$ is a $T_1$ represented topological space, then the thus obtained space $\K(X)$ coincides with the set of all compact subsets of the \emph{sequentialisation} $\operatorname{seq}(X)$ of $X$, whose open sets are the sequentially open sets of $X$. For details see \cite{SchroederPhD}. By Proposition 3.3.2 (3) in \cite{SchroederPhD}, the notions of compactness, sequential compactness, and compactness in the sequentialisation coincide for Hausdorff represented topological spaces, so we obtain Definition \ref{Def: canonical representations of compact sets}. 

It will sometimes be convenient to work with an intrinsic notion of computable compactness for represented spaces, which we introduce next.
\begin{Definition}\label{Def: compact space}
A Hausdorff represented topological space $X$ is called \demph{computably compact}, if the mapping
\[\operatorname{empty?}\colon \A(X) \to \S,\; A \mapsto \begin{cases}1 &\text{ if }A = \emptyset, \\0 &\text{ otherwise.} \end{cases} \]
is computable.
\end{Definition}

Note that the terminology used in Definitions \ref{Def: compact space} and \ref{Def: canonical representations of compact sets} is consistent in the sense that $X$ is a computably compact space if and only if $X$ is a computable point in $\K(X)$. The next proposition is a converse to this in some sense.

\begin{proposition}\label{prop: compact subspace of represented space is compact set}
Let $X$ be a Hausdorff represented topological space and $K \subseteq X$ be a nonempty co-semi-decidable subset, such that the represented space $(K,\delta_X\big|^K)$ is computably compact. Then $K$ is a computably compact subset of $X$, i.e.\ a computable point in $\K(X)$.
\end{proposition}
\begin{proof}
It follows immediately from the definition of $\O(X)$ that the mapping
\[\cap_K \colon \O(X) \to \O(K),\; U \mapsto U \cap K\]
is computable. Now, $U \supseteq K$ if and only if $K \setminus (U \cap K) = \emptyset$. It again follows from the definition, that the mapping
\[\O(K) \to \A(K), \; U \mapsto K\setminus U \]
is computable. Since $K$ is a computably compact represented space, the mapping
\[\A(K) \to \S, \; A \mapsto \begin{cases}1 &\text{ if }A = \emptyset, \\ 0 &\text{ if }A \neq \emptyset\end{cases}\]
is computable. It follows that the set of open subsets of $X$ containing $K$ is computably open, i.e. $K$ is computably compact.
\end{proof}

A closed subset of a compact space is compact, and in a Hausdorff space, every compact set is closed. We have an effective counterpart of this in the theory of represented spaces. A represented space $X$ is called \emph{effectively Hausdorff} if the mapping $X \to \A(X),\; x \mapsto \{x\}$ is computable.

\begin{proposition}\label{prop: closed and compact sets, (i) closed set is compact, (ii) compact set in Hausdorff space is closed}
Let $X$ be a Hausdorff computably compact represented topological space.
\begin{enumerate}[label=(\roman*)]
\item The mapping $\id\colon \A(X) \to \K(X)$ is well-defined and computable.
\item If $X$ is effectively Hausdorff, then the mapping $\id\colon \K(X) \to \A(X)$ is well-defined and computable.
\end{enumerate}
\end{proposition}
\begin{proof}\hfill
\begin{enumerate}[label=(\roman*)]
\item We are given a closed set $A \in \A(X)$ which we want to compute as a compact set $A \in \K(X)$. Given an open set $U \in \O(X)$ we want to verify if $U \supseteq A$. In order to do so, check if $U \cup A^C$ covers $X$, using that $X$ is computably compact.
\item We are given a compact set $K \in \K(X)$ which we want to compute as a closed set $K \in \A(X)$. Given a point $x \in X$ we want to verify if $x \notin K$. In order to do so, compute $\{x\} \in \A(X)$, using that $X$ is effectively Hausdorff, and verify if $\{x\}^C \supseteq K$, using the compactness information on $K$.\qedhere
\end{enumerate}
\end{proof}

\noindent It is easy to see that the computability and well-definedness of the mapping \[{\id\colon \A(X) \to \K(X)}\] \emph{characterises} computably compact represented spaces (cf.~also \cite{PaulyRepresented}).

\begin{theorem}\label{thm: supremum on compact space}
Let $K$ be a computably compact represented topological space, containing a computable dense sequence. Then the mapping
\[\max\colon \RCo{K}{\R} \to \R,\; f \mapsto \max\{f(x)\;\big|\; x \in K\}\]
is well-defined and computable.
\end{theorem}
\begin{proof}
Since $K$ is adequately represented, we have $\RCo{K}{\R} = \Co_{\operatorname{seq}}(K,\R)$, and $K$ is sequentially compact thanks to Proposition 3.3.2 (3) in \cite{SchroederPhD}. It follows that for any $f \in \RCo{K}{\R}$, the set $f(K)$ is sequentially compact in $\R$ and thus compact. This shows that $\max$ is well-defined. It remains to show that $\max(f)$ is computable relative to $f$. Let $(x_n)_n$ be a computable dense sequence in $K$. Then the sequence $(f(x_n))_n$ is computable relative to $f$, with $\sup_{n \in \N} f(x_n) = \max(f)$. On the other hand, for every computable $b \in \R$, the set $U_b = \{x \in K\;\big|\; f(x) < b\}$ is semi-decidable relative to $f$, so that by the computable compactness of $K$, the predicate $\forall x \in K.\left(f(x) < b\right)$ is semi-decidable relative to $f$ for all $b \in \Q$. We can use this to construct a sequence $(b_n)_n$ of real numbers which is computable relative to $f$ and satisfies $\max(f) = \inf_{n \in \N} b_n$. Since $\max(f)$ can hence be approximated arbitrarily well ``from above'' as well as ``from below'', it is computable relative to $f$.
\end{proof}

It follows from Theorem \ref{thm: supremum on compact space} that every finite dimensional uniformly convex computable Banach space $E$ has a computable modulus of uniform convexity $\eta_E$, since we may put
\[\eta_E(\varepsilon) = \inf\Big\{1 - \bignorm{\frac{x + y}{2}} \;\big|\; x,y \in B_E, \norm{x - y} \geq \varepsilon  \Big\}, \]
and the set $\{(x,y) \in B_E\times B_E \;\big|\; \norm{x - y}\geq \varepsilon\}$ is computably compact relative to $\varepsilon$ and contains a computable, dense sequence relative to $\varepsilon$. Since the proof of Theorem \ref{thm: supremum on compact space} is uniform in $K$, the claim follows. 
\begin{theorem}[Kreinovich's theorem, \cite{KreinovichPhd}]\label{thm: Kreinovich's theorem}
Let $K$ be a computably compact represented topological space. Then the mapping
\[\UC_K\colon\subseteq \A(K) \to K, \; \{x\}\mapsto x \]
is computable.\qed
\end{theorem}

\begin{Definition}
A \emph{computable metric space} is a triple $(M,d,\nu_M)$, where $(M,d)$ is a metric space and $\nu_M\colon\N\to A$ is a numbering of a dense subset $A\subseteq M$, such that $d\colon A\times A \to \R$ is $(\nu_M\times\nu_M,\rho)$-computable. With a computable metric space we associate the represented space $(M,\delta_M)$, where 
\[\delta_M(p) = x \;:\Leftrightarrow d(\nu_M(p(n)),x) \leq 2^{-n} \;\text{ for all }n \in \N.\]
\end{Definition}
One can show that the above defined canonical representation of a computable metric space $M$ is admissible and that $d\colon M\times M\to\R$ is computable. This canonical representation $\delta_M$ is also called the \emph{Cauchy representation} induced by $\nu_M$. We will refer to the points in $\im \nu_M$ as the \emph{rational points} of the represented space $M$. Note that any computable metric space is separable, and hence Hausdorff, by definition. In fact, every computable metric space is effectively Hausdorff, since the predicate $d(x,y) > 0$ is a semi-decidable relative to $x$ and $y$. In any metric space $M$ we denote by $B(x,r) = \{y \in M \;\big|\; d(x,y) < r\}$ the open ball of radius $r$ centred at $x$, and by $\clos{B}(x,r) = \{y \in M \;\big|\; d(x,y) \leq r\}$ the closed ball of radius $r$ centred at $x$.

The following result is more or less folklore, and justifies the more abstract Definition \ref{Def: canonical representations of compact sets} of computable compactness.
\begin{proposition}\label{prop: computable total boundedness equivalent computable compactness}
A computable metric space is computably compact if and only if it is complete and \emph{computably totally bounded}, i.e. if and only if there exists a function $\alpha\colon \N\to M^{*}$ such that 
\[\forall n \in \N. \forall x \in M.\exists k \leq \operatorname{lth}(\alpha(n)).\left(d(x,\alpha(n)_k) < 2^{-n} \right). \]
We call $\alpha(n)$ a $2^{-n}$-net in $M$.
\end{proposition}
\begin{proof}
Suppose that $M$ is computably compact. Then $M$ is compact and thus complete. Let $(a_k)_k$ be a dense computable sequence in $M$. Since $M$ is computably compact, we can verify for all $n,m \in \N$ if the open set $B(\tilde{a}_1,2^{-n - 1})\cup\dots\cup B(\tilde{a}_m,2^{-n - 1})$ is equal to all of $M$, where $\tilde{a}_k$ is a rational approximation to $a_k$ to up error $2^{-n-1}$. In that case, $a_1,\dots,a_m$ is a $2^{-n}$-net in $M$. On the other hand, since $(a_k)_k$ is dense in $M$ and $M$ is compact, this process has to finish after a finite number of steps for each $n\in\N$. It follows that $M$ is computably totally bounded.

Suppose now that $M$ is complete and computably totally bounded. Let \[S = \{B(c_1,r_1),\dots,B(c_k,r_k)\}\] be a collection of rational balls, i.e.~balls whose radii are rational numbers and whose centres are rational points in $M$. We show that we can verify if $S$ is a cover of $A$. It is a standard argument that this suffices in order to establish that $M$ is computably compact. Let $(\langle a^n_1,\dots,a^n_{l(n)}\rangle)_{n \in \N}$ be a computable sequence of $2^{-n}$-nets in $M$. Then $S$ covers $M$ if and only if there exists $n\in \N$ such that for all $i \in \{1,\dots,l(n)\}$ there exists $j \in \{1,\dots,k\}$ such that 
$d(a^n_i,c_j) < r_j - 2^{-n}$. This property is semi-decidable, so $M$ is computably compact.
\end{proof}

In complete computable metric spaces, overtness is characterised by separability, in the sense that Proposition \ref{prop: overtness and recursive enumerability} admits a converse (cf. \cite[Theorem 3.8 (2)]{BrattkaPresser}).
\begin{proposition}\label{prop: overtness and recursive enumerability, metric case}
Let $M$ be a complete computable metric space. Define the representation $\delta_{\re}$ of the set of nonempty closed subsets of $M$ as in Proposition \ref{prop: overtness and recursive enumerability}. Then ${\delta_{\re} \equiv \left(\upsilon\big|^{\V(M)\setminus\{\emptyset\}}\right)}$.
\end{proposition}
\begin{proof}
The direction $\delta_{\re} \leq \upsilon$ was already proved in Proposition \ref{prop: overtness and recursive enumerability}, so it remains to prove $\upsilon \leq \delta_{\re}$. Suppose we are given a closed set $A \in \V(M)$. We can compute an enumeration $(B_m)_m$ of all open rational balls (i.e. balls with rational centre and radius) with radius at most $1$ intersecting $A$. We use this to construct a dense sequence $(x_m)_m$ in $A$. The $m^{th}$ element in the sequence is computed as follows: the first approximation $x_m^{(0)}$ to $x_m$ is the centre of $B_m$. Let $1 \geq \varepsilon > 0$ denote the radius of $B_m$. We claim that we can find an open rational ball $B_m^{(1)}$ with radius at most $\tfrac{\varepsilon}{2}$ which is contained in $B_m$ and intersects $A$. Let $a \in B_m\cap A$. Then there exists rational $\tfrac{\varepsilon}{2} > \delta > 0$ such that $d(a,x_m^{(0)}) < \varepsilon - \delta$. Let $\tilde{a}$ be a rational approximation of $a$ up to error $\delta/2$. Then $d(\tilde{a},x_m^{(0)}) < \varepsilon - \delta/2$. In particular, $B(\tilde{a},\delta/2)\cap A \neq \emptyset$ and $B(\tilde{a},\delta/2) \subseteq B(x_m^{(0)},\varepsilon)$. On the other hand, we can verify for a given rational $a$ and $\delta < \tfrac{\varepsilon}{2}$ that $d(a,x_m^{(0)}) < \varepsilon - \delta$ and that $B(a,\delta) \cap A \neq \emptyset$. We may hence search for such $a$ and $\delta$, and put $B_m^{(1)} = B(a,\delta)$ and $x_m^{(1)} = a$. Continuing in this manner, we obtain a Cauchy sequence $x_m^{(n)}$ with $d(x_m^{(n)},x_m^{(n + k)}) < 2^{-n}$ for all $k,n \in \N$. Since $M$ is complete, the sequence $x_m^{(n)}$ converges to some element $x_m \in A$ with $d(x_m^{(n)},x_m) \leq 2^{-n}$. Applying this to all $(B_m)_m$ in parallel, we obtain a computable sequence $(x_m)_m$. It remains to show that $(x_m)_m$ is dense in $A$. Let $a \in A$, and let $\varepsilon$ be a rational number satisfying $1 > \varepsilon > 0$. There exists a rational point $x$ satisfying $d(a,x) < \tfrac{\varepsilon}{2}$. In particular, $B(x,\tfrac{\varepsilon}{2}) \cap A \neq \emptyset$, so $B(x,\tfrac{\varepsilon}{2}) = B_k$ for some $k \in \N$. It follows from the construction of $x_k$ that $d(x_k,a) < \varepsilon$.
\end{proof}

If $M$ is a computable metric space, we have another natural notion of computability for closed sets, by identifying a closed set $A \subseteq M$ with its \emph{distance function}
\begin{equation}\label{eq: distance}
d_A\colon M\to \R, \; d_A(x) = \inf\{d(x,y) \;|\; y \in A\}.
\end{equation}

\noindent Define represented spaces $\R_{<} = (\R,\rho_{<})$ and $\R_{>} = (\R,\rho_{>})$ via
\[\rho_{<}(p) = x\; :\Leftrightarrow\; \sup_{n \in \N}\nu_{\Q}^{\omega}(p)(n) = x \;\;\text{ and }\;\;\rho_{>}(p) = x\; :\Leftrightarrow\; \inf_{n \in \N}\nu_{\Q}^{\omega}(p)(n) = x.\]
Computable elements of $\R_{<}$ are called \emph{left-r.e.\ numbers}, and computable elements of $\R_{>}$ are called \emph{right-r.e.\ numbers}. Obviously, a number is computable if and only if it is both right- and left-r.e., whereas a classic result due to Specker \cite{SpeckerSequence} asserts the existence of both uncomputable left-r.e.- and uncomputable right-r.e.\ numbers.

\begin{Definition}\label{Def: distance representations}
Let $M$ be a computable metric space.
\begin{enumerate}[label=(\roman*)]
\item The represented space $\A_{\dist}(M)$ is the space of nonempty closed subsets of $M$, where a closed subset $A\subseteq M$ is represented via a $[\delta_M\to\rho]$-name of its distance function \eqref{eq: distance}.
\item The represented space $\A_{\dist_{<}}(M)$ is the space of nonempty closed subsets of $M$, where a closed subset $A\subseteq M$ is represented via a $[\delta_M\to\rho_<]$-name of its distance function \eqref{eq: distance}.
\item The represented space $\A_{\dist_{>}}(M)$ is the space of nonempty closed subsets of $M$, where a closed subset $A\subseteq M$ is represented via a $[\delta_M\to\rho_>]$-name of its distance function \eqref{eq: distance}.
\end{enumerate}
\end{Definition}

\noindent Computable points of $\A_{\dist}(M)$ are called \demph{located}, computable points of $\A_{\dist_<}(M)$ are called \demph{lower semi-located}, and computable points of $\A_{\dist_>}(M)$ are called \demph{upper semi-located}. Using Proposition \ref{prop: overtness and recursive enumerability, metric case}, it is easy to see that for any complete computable metric space $M$, the canonical representations of the spaces $\A_{\dist_>}(M)$ and $\V(M)\setminus\{\emptyset\}$ are equivalent (see also \cite[Theorem 3.7]{BrattkaPresser}). It is also easy to see that $\id\colon \A_{\dist_{<}}(M) \to \A(M)$ is computable (see \cite[Theorem 3.11 (1)]{BrattkaPresser}). In \cite[Lemma 5.1.7]{Weih} it is proved, that for $M = \R^d$, the canonical representations of $\A_{\dist_<}(M)$ and $\A(M)\setminus\{\emptyset\}$ are equivalent, and the argument readily generalises to any complete computable metric space with (effectively) compact closed balls (see \cite[Theorem 3.11 (3)]{BrattkaPresser}). Any such space is locally compact. Local compactness is in fact necessary for the reduction to hold:

\begin{proposition}\label{prop: equivalence lower semi-located and computably closed implies locally compact}
Let $M$ be a complete computable metric space. If the identity mapping $\id\colon \A(M)\setminus\{\emptyset\} \to \A_{\dist_<}(M)$ is computable, then $M$ is locally compact.
\end{proposition}
\begin{proof}
If $M$ is a singleton, the claim is trivial, so we may assume that $M$ consists of at least two points. Given $x \in M$ we show that we can compute $\clos{B}(x,r)$ as a compact subset of $M$, for $r$ sufficiently small. We search for a rational $y \in M$ and $r\in\Q_{+}$ with $d(y,x) > r$. By an argument similar to Proposition \ref{prop: computable total boundedness equivalent computable compactness}, it suffices to compute for every $k \in \N$ a cover of $\clos{B}(x,r)$ by balls of radius $2^{-k}$, whose centres are rational points in $M$. Since $y$ is computable, the singleton $\{y\}$ is co-semi-decidable, so we can compute an enumeration of balls with rational centres and radii exhausting the complement of $\{y\}$ such that every ball has radius at most $2^{-k}$. We feed this enumeration into the machine computing the identity $\id\colon \A(M)\setminus\{\emptyset\} \to \A_{\dist_<}(M)$ and use the $[\delta_M\to\rho_<]$-name provided by the machine to compute $\d(x,\{y\})$ from below. After having processed finitely many balls, the machine will output the lower bound $r$ on $\d(x,\{y\}) = d(x,y)$. This means that $\clos{B}(x,r)$ is covered by these finitely many balls, since otherwise we could force the machine computing the identity to err.
\end{proof}

The proof of Proposition \ref{prop: equivalence lower semi-located and computably closed implies locally compact} shows that we can even compute a witness for the local compactness of $M$, namely a function $f\colon M \to \K(M)$ which maps a point $x$ to a compact closed ball containing $x$. 

\begin{Definition}
A \emph{(real) computable normed space} is a normed real vector space $E$ together with a numbering $e\colon\N\to E$ such that $\spann\{e(n) \;|\; n \in \N\}$ is dense in $E$ and $(E,d,\nu_E)$ is a computable metric space, where $d(x,y) = \norm{x - y}$ and $\nu_E$ is a canonical notation of all (finite) $\Q$-linear combinations of $\im e$.
\end{Definition}

A complete computable normed space is called a computable Banach space. A computable normed space which is also a Hilbert space is called a computable Hilbert space. The inner product in a computable Hilbert space is computable by the polarisation identity. A computable normed space becomes a represented space when endowed with the Cauchy representation induced by the numbering $\nu_E$. In this representation the vector space operations and the norm are computable functions and $0 \in E$ is a computable point of the represented space $E$. An important feature of (infinite dimensional) computable normed spaces is that without loss of generality the fundamental sequence is linearly independent (cf.~\cite[p.~142]{PourElRichards}).
\begin{lemma}[Effective independence lemma]\label{lem: effective independence lemma}
Let $(E,e)$ be an infinite dimensional computable normed space. Then there exists a computable function $f\colon \N \to \N$ such that $e\circ f\colon\N\to E$ has dense span in $E$ and consists of linearly independent vectors.\qed
\end{lemma}

\begin{corollary}\hfill
\begin{enumerate}[label=(\roman*)]
\item Every computable real Hilbert space $H$ has a computable orthonormal basis, i.e. an orthonormal basis which is a (potentially finite) computable sequence in $H$.
\item Every finite dimensional computable real Hilbert space is computably isometrically isomorphic to $\R^d$ for some $d \in \N$. Every infinite dimensional real Hilbert space is computably isometrically isomorphic to $\ell^2$.\qed
\end{enumerate}
\end{corollary}

\noindent Let us now introduce some basic notions from the theory of Weihrauch degrees. We will treat this paragraph somewhat informally, as we will not need to develop the theory very far. A formal and comprehensive treatment of everything stated here can be found in \cite{WeihrauchDegrees, ClosedChoice, EffectiveChoice}, and in \cite{ConnectedChoice}, where the Weihrauch degree of Brouwer's fixed point theorem is determined.
\begin{Definition}
Let ${\tuple{\cdot}{\cdot}\colon \N^\N\times\N^\N \to \N^\N}$ denote some computable pairing function on Baire Space. A multimapping ${g\colon X \rightrightarrows Y}$ between represented spaces $X$ and $Y$ is said to \emph{Weihrauch reduce} to ${h\colon Z \rightrightarrows W}$, in symbols ${g \leq_W h}$, if there exist computable functions ${K,N\colon\subseteq \N^\N \to \N^\N}$ such that ${K\tuple{HN}{\id}}$ is a realiser of $g$, whenever $H$ is a realiser of $h$. If $f \leq_W g$ and $g \leq_Wf$ we say that $g$ and $f$ are \emph{Weihrauch equivalent} and write $f \equiv_W g$. The equivalence classes with respect to $\equiv_W$ are called \emph{Weihrauch degrees}.
\end{Definition}
The Weihrauch degrees together with the $\leq_W$-relation are known to form a bounded lattice. A very important and useful tool for studying Weihrauch degrees are so-called \emph{closed choice principles} on represented spaces.
\begin{Definition}
Let $X$ be a represented space. 
\begin{enumerate}[label=(\roman*)]
\item The \emph{closed choice principle} on $X$ is the multimapping
\[\C_X\colon \A(X)\setminus\{\emptyset\} \rightrightarrows X,\; A \mapsto A.\]
The \emph{unique choice principle} $\UC_X$ on $X$ is $\C_X$ restricted to singleton sets and the \emph{connected choice principle} $\CC_X$ is $\C_X$ restricted to connected sets.
\item Let $X$ additionally be a closed subset of a computable Banach space $E$. We define the \emph{convex choice principle} $\XC_X$ as the restriction of $\C_X$ to the space $\A^{\conv}(X)$ of convex closed subsets of $X$.
\end{enumerate}
\end{Definition}

\noindent Let us now introduce some concrete Weihrauch degrees that will be useful in our further studies. The limit operator $\lim\colon\subseteq\N^\N\to\N^\N$ takes as input a (suitably encoded) convergent sequence $(p_n)_n \in (\N^\N)^\N \simeq \N^\N$ and outputs its limit. Weak K\H{o}nig's Lemma, $\WKL$, takes as input an infinite binary tree and outputs an infinite path. The intermediate value theorem, $\IVT$, takes as input a continuous function $f\colon [0,1] \to \R$ with $f(0)\cdot f(1) < 0$ and outputs some point $x \in [0,1]$ such that $f(x) = 0$. Brouwer's fixed point theorem in $n$-dimensional space, $\BFT_n$, takes as input a continuous function $f\colon [0,1]^n\to[0,1]^n$ and outputs some fixed point of $f$. Their relation is summarised in the following
\begin{fct}\label{fact on Weihrauch degrees}\hfill
\begin{enumerate}[label=(\roman*)]
\item $\WKL \equiv_W \C_{\{0,1\}^\omega}$.
\item $\CC_{[0,1]^n} \equiv_W \BFT_n \leq_W \WKL$ for all $n$.
\item $\IVT \equiv_W \XC_{[0,1]} \equiv_W \CC_{[0,1]} \equiv_W \BFT_1 <_W \BFT_2 \leq_W \BFT_3 \equiv_W \WKL$.
\item $\WKL <_W \lim <_W \C_{\N^\N}$.\qed
\end{enumerate}
\end{fct}

\noindent An important property of computably compact spaces is that their closed choice principle is of low degree.
\begin{theorem}\label{thm: compact choice below WKL}
Let $K$ be a computably compact represented topological space. Then the multimapping
\[\C_K\colon \A(K)\setminus\{\emptyset\} \rightrightarrows K, \; A\mapsto A \]
satisfies $\C_K \leq_W \WKL$.\qed
\end{theorem}

Finally, we need a few observations from (computable) functional analysis. The first theorem is the so-called \demph{projection theorem}, which can be found in virtually any functional analysis textbook (cf.~e.g.~\cite[Satz V.3.2 \& Lemma V.3.3]{Werner}).
\begin{theorem}\label{thm: projection theorem}\hfill
\begin{enumerate}[label=(\roman*)]
\item Let $E$ be a uniformly convex real Banach space and let $K \subseteq E$ be nonempty, closed, and convex. For every $x \in E$ there exists a unique $y \in K$ such that \[d(x,y) = d(x,K) = \inf \{d(x,z)\;|\; z\in K\}.\] We denote this element by $P_K(x)$. The mapping $P_K$ is a continuous retraction onto $K$, called the \emph{metric projection}.
\item Let $H$ be a real Hilbert space and $K \subseteq H$ be nonempty, closed, and convex. Then $P_K$ is a nonexpansive mapping, and for all $x \in H$ the element $P_K(x)$ is characterised by the \emph{variational inequality}
\[\product{x - P_K(x)}{y - P_K(x)} \leq 0 \;\;\text{for all }y \in K.\]
\end{enumerate}
\end{theorem}
\proof
\hfill
\begin{enumerate}[label=(\roman*)]
\item We may assume that $x \notin K$ and $x = 0$. Put $r = \inf\{\norm{z} \;\big|\; {z \in K}\} > 0$.

\underline{Existence:}  Let $(y_n)_n$ be a sequence in $K$ with $\lim_{n \to \infty}\norm{y_n} = r$. We show that $(y_n)_n$ is a Cauchy sequence. Let $\varepsilon > 0$. There exists $n \in \N$ such that $\norm{y_{n + k}} < r + \varepsilon$ for all $k \geq 0$. We have $\tfrac{y_n + y_{n + k}}{2} \in K$ for all $k \geq 0$ and thus $\norm{\tfrac{y_n + y_{n + k}}{2}} \geq r$. If $\delta \in (0,1]$ satisfies $\eta(\delta) > \tfrac{\varepsilon}{r + \varepsilon}$, then
\[\bignorm{\frac{y_n + y_{n + k}}{2}} \geq r > \left(r + \varepsilon\right)\left(1 - \eta(\delta)\right).\]
Applying the contraposition of uniform convexity to $\tfrac{y_{n + k}}{r + \varepsilon}$ and $\tfrac{y_n}{r + \varepsilon}$ thus yields
\[\norm{y_n - y_{n + k}} < \left(r + \varepsilon\right) \delta. \]
Since $\tfrac{\varepsilon}{r + \varepsilon} \to 0$ as $\varepsilon \to 0$, and $\eta(\delta) > 0$ for all $\delta \in (0,1]$, it follows that $\norm{y_n - y_{n + k}} \to 0$ as $n \to \infty$, i.e. the sequence $(y_n)_n$ is a Cauchy sequence. Since $K$ is closed and $E$ is complete, it converges to some element $y \in K$, which satisfies $\norm{y} = \lim_{n\to\infty}\norm{y_n} = r$.

\underline{Uniqueness:} Suppose that the points $y_1, y_2 \in K$ with $y_1 \neq y_2$ satisfy $\norm{y_i} = r > 0$. Then we have $y_1, y_2 \in \clos{B}(0,r)$. Since $K$ is convex, $\tfrac{y_1 + y_2}{2} \in K$ and since $E$ is strictly convex and $y_1 \neq y_2$, we have $\norm{\tfrac{y_1 + y_2}{2}} < r$. Contradiction.
\item On one hand we have for all $\alpha \in [0,1]$ and $z \in K$:
\begin{align*}
\norm{x - P_K(x)}^2 &\leq \norm{x - \left(\alpha z + (1 - \alpha) P_K(x)\right)}^2 \\
&= \product{x - P_K(x) - \alpha(z - P_K(x))}{x - P_K(x) - \alpha(z - P_K(x))}\\
&= \norm{x - P_K(x)}^2 - 2\alpha\product{x - P_K(x)}{z - P_K(x)} + \alpha^2\norm{z - P_K(x)}^2
\end{align*}
and thus $\product{x - P_K(x)}{z - P_K(x)} \leq \frac{\alpha}{2} \norm{z - P_K(x)}^2$ for all $\alpha \in (0,1]$, which shows that $P_K$ satisfies the variational inequality. On the other hand, if $p \in K$ satisfies $\product{x - p}{z - p} \leq 0$ for all $z \in K$, then for all $z \in K$ we have
\begin{align*}
\norm{x - z}^2 &= \norm{(x - p) + (p - z)}^2 \\
&= \norm{x - p}^2 + 2\product{x - p}{p - z} + \norm{p - z}^2\\
&\geq \norm{x - p}^2
\end{align*}
and thus $p = P_K(x)$ by Theorem \ref{thm: projection theorem}. 
It remains to show that $P_K$ is nonexpansive. Let $x,y \in H$. We may assume that $x \neq y$ and $P_K(x) \neq P_K(y)$. Since $P_K(x), P_K(y) \in K$, we may use the variational inequality to obtain
\[\product{P_K(y) - P_K(x)}{x - P_K(x)} \leq 0 \]
and
\[\product{P_K(x) - P_K(y)}{y - P_K(y)} \leq 0.\]
Adding both inequalities yields
\[\product{P_K(y) - P_K(x)}{x - y + P_K(y) - P_K(x)} \leq 0.\]
And hence
\[\norm{P_K(y) - P_K(x)}^2 \leq \product{P_K(x) - P_K(y)}{x - y} \leq \norm{P_K(x) - P_K(y)}\cdot\norm{x - y},\]
where the last inequality is the Cauchy-Schwarz inequality. We thus obtain
\[\norm{P_K(x) - P_K(y)} \leq \norm{x - y}.\eqno{\qEd}\]
\end{enumerate}\medskip

\noindent The next important result is that projections onto located convex sets in uniformly convex computable Banach spaces are computable relative to a modulus of convexity. This will follow from a highly uniform proof mining result due to Kohlenbach:
\begin{theorem}[{\cite[Proposition 17.4]{Kohlenbook}}]\label{thm: modulus of uniqueness for projection}
There exists a computable functional 
\[\Phi\colon \N^{\N}\times\N\times\N \to \N \]
such that if $E$ is a uniformly convex normed space with modulus of uniform convexity $\eta_E$, $\mu$ is any functional satisfying
$2^{-\mu(n)} \leq \eta_E(2^{-n})$, $K \subseteq E$ is nonempty, closed, and convex, and $x \in E$ with $d(x,K) \leq d$ then
\[\phi(n) = \Phi(\mu,d,n) \]
is a modulus of uniqueness for the projection onto $K$. This means that if $p, q \in K$ satisfy
$\norm{p - x} \leq d(x,K) + 2^{-\phi(n)}$
and 
$\norm{q - x} \leq d(x,K) + 2^{-\phi(n)}$,
then $\norm{p - q} < 2^{-n}$.\qed
\end{theorem}

\begin{corollary}\label{cor: computability of projection on subset}
Let $E$ be a uniformly convex computable Banach space, let $C \subseteq E$ be nonempty, convex, and computably overt. Let $\eta_E$ be a modulus of uniform convexity for $E$. Let $\A^{\conv}_{\dist}(C)$ denote the represented space of nonempty convex closed subsets of $C$, represented via their distance function. Then the mapping
\[P\colon \A^{\conv}_{\dist}(C) \to \Co(C,C),\; K \mapsto P_K, \]
is computable relative to $\eta_E$. In fact it is computable relative to any $\mu\colon\N\to\N$ satisfying $2^{-\mu(n)} \leq \eta_E(2^{-n})$.
\end{corollary}
\begin{proof}
Let $\mu\colon\N\to\N$ be such that $2^{-\mu(n)} \leq \eta_E(2^{-n})$ and let $\Phi$ be the functional from Theorem \ref{thm: modulus of uniqueness for projection}. We are given a set $K \in \A^{\conv}_{\dist}(C)$, a point $x \in C$ and a number $n \in \N$ and want to compute an approximation to $P_K(x)$ up to error $2^{-n}$. Since we are given the distance function to $K$, we can compute an integer upper bound $d$ to $d(x,K)$. Again using the distance function, we can compute a dense sequence in $K$. This allows us to find a point $p \in K$ with $\norm{p - x} \leq d(x,K) + 2^{-\Phi(\mu,d,n)}$. It follows from Theorem \ref{thm: modulus of uniqueness for projection} that $\norm{p - P_K(x)} < 2^{-n}$.
\end{proof}\medskip

The special case where $E$ has a computable modulus of convexity and $C = E$ yields:

\begin{corollary}[\cite{BrattkaDillhage}]\label{cor: computability of projection on E}
Let $E$ be a uniformly convex computable Banach space with computable modulus of uniform convexity. Then the mapping
\[P\colon \A^{\conv}_{\dist}(E) \to \Co(E,E),\; K \mapsto P_K, \]
is computable. In particular, if $K \subseteq E$ is a nonempty located and convex set, then $P_K$ is $(\delta_E,\delta_E)$-computable.\qed
\end{corollary}

The (for our purpose) most important structural feature of fixed point sets of nonexpansive mappings in strictly convex Banach spaces is that they are always convex. This is a standard exercise in functional analysis. We will prove it here anyway, to give a simple example of a proof exploiting the convexity of the underlying space.
\begin{proposition}\label{prop: fixed point sets of nonexpansive mappings are convex}
Let $E$ be a strictly convex Banach space, let $K \subseteq E$ be nonempty, closed, bounded, and convex. Let $f\colon K \to K$ be nonexpansive. Then the set $\Fix(f)$ is convex.
\end{proposition}
\begin{proof}
Since $f$ is continuous, it suffices to show that for each $x,y \in \Fix(f)$, the convex combination $\frac{x + y}{2}$ is again contained in $\Fix(f)$ (since then it follows that the set of dyadic convex combinations of $x$ and $y$, which is dense in the line segment joining $x$ and $y$, consists entirely of fixed points). Since $f$ is nonexpansive, we have
\[\bignorm{f\left(\frac{x + y}{2}\right) - x} = \bignorm{f\left(\frac{x + y}{2}\right) - f(x)} \leq \bignorm{\frac{x - y}{2}}.\]
Similarly, $\norm{f(\tfrac{x + y}{2}) - y} \leq \tfrac{1}{2}\norm{x - y}$ and obviously the same inequality holds if we replace $f(\tfrac{x + y}{2})$ by $\tfrac{x + y}{2}$. Now, suppose that $a, b \in E$ satisfy 
\[\norm{a - y} \leq \bignorm{\frac{x - y}{2}},\;\norm{a - x} \leq \bignorm{\frac{x - y}{2}},
\;\norm{b - x} \leq \bignorm{\frac{x - y}{2}},\;\norm{b - y} \leq \bignorm{\frac{x - y}{2}}.\]
Then, if $a \neq b$, strict convexity yields
\[\bignorm{\frac{a + b}{2} - y} < \bignorm{\frac{x - y}{2}}\;\text{ and }\;\bignorm{\frac{a + b}{2} - x} < \bignorm{\frac{x - y}{2}}\]
and hence
\[\norm{x - y} \leq \bignorm{x - \frac{a + b}{2}} + \bignorm{y - \frac{a + b}{2}} < \norm{x - y}, \]
contradiction. It follows, that $f(\tfrac{x + y}{2}) = \tfrac{x + y}{2}$.
\end{proof}

\section{Computability of Fixed points and Rates of Convergence}

In this section we study the computability-theoretic complexity of the problems of finding fixed points of nonexpansive mappings on compact domains, and of obtaining rates of convergence of certain fixed point iterations. Let us first state some natural computational problems associated with the fixed point properties of nonexpansive mappings and determine their rough relation.

\begin{Definition}\label{Def: fixed point iteration}
Let $E$ be a uniformly convex real Banach space, let $K \subseteq E$ be a nonempty subset of $K$, and let $\Ne(K)$ denote the set of nonexpansive self-maps of $K$. A \emph{fixed point iteration} on $K$ is a mapping 
$I\colon \Ne(K)\times K \to K^{\N}$
such that for all $f \in \Ne(K)$, $x \in K$, we have $\lim_{n \to \infty} I(f,x)(n) \in \Fix(f)$.
\end{Definition}

\begin{Definition}\label{Def: computational problems}
Let $E$ be a uniformly convex computable Banach space, let $K \subseteq E$ be nonempty, co-semi-decidable, computably overt, bounded, and convex. Let $\Ne(K)$ be the represented space of nonexpansive self-maps of $K$ with representation $[\delta_K\to\delta_K]\Big|^{\Ne(K)}$. Let $I\colon \Ne(K) \times K \to K^{\N}$ be a computable fixed point iteration. Consider the following computational problems:
\begin{enumerate}[label=(\roman*)]
\item The \emph{realiser problem for the Browder-Göhde-Kirk theorem} $\BGK_K$: Given a nonexpansive function $f\colon K \to K$, output a fixed point for $f$. More formally:
\[\BGK_K\colon \Ne(K) \rightrightarrows K,\; f \mapsto \Fix(f).\]
\item The \emph{projection problem} $\Proj_K$: Given a nonexpansive function $f\colon K \to K$, and a point $x\in K$ output the metric projection of $x$ onto $\Fix(f)$. More formally:
\[\Proj_K\colon \Ne(K)\times K \to K,\; (f,x) \mapsto P_{\Fix(f)}(x).\]
\item The \emph{limit problem} $\lim(I)$ for $I$: given a nonexpansive function ${f\colon K \to K}$ and a starting point $x \in K$, output $\lim_{n \in \N} I(f,x)(n)$. More formally:
\[\lim(I)\colon \Ne(K)\times K \to K,\; (f,x) \mapsto \lim_{n\to \infty} I(f,x)(n).\]
\item The \emph{rate of convergence problem} $\Conv_{I}$ for $I$: given a nonexpansive function ${f\colon K \to K}$ and a starting point $x \in K$, output a rate of convergence of the sequence $(I(f,x)(n))_n$. More formally:
\begin{align*}
\Conv_I\colon& \Ne(K)\times K \rightrightarrows \N^\N,\\
&(f,x) \mapsto \{\varphi\in\N^\N\;|\; \forall n \in \N.\forall l \geq \varphi(n). \norm{I(f,x)(l) - \lim_{k \to \infty}I(f,x)(k)} < 2^{-n}\}.
\end{align*}
\end{enumerate}
\end{Definition}

\noindent Most fixed point iterations considered in the literature are of a far more particular form than just computable mappings. This can be exploited to obtain stronger uncomputability results for particular classes of fixed point iterations. We summarise some common properties.

\begin{Definition}
Let $E$ be a uniformly convex real Banach space, and let $K \subseteq E$ be nonempty, convex, closed, and bounded. Let $I\colon \Ne(K)\times K \to K^{\N}$ be a fixed point iteration.
\begin{enumerate}[label=(\roman*)]
\item $I$ is called \demph{projective} if for all $f\in \Ne(K)$ and $x \in K$, the limit $\lim_{n \to \infty}I(f,x)(n)$ is the unique fixed point of $f$ which is closest to $x$.
\item $I$ is called \demph{retractive} if for all $f\in\! \Ne(K)$ and $x \in \Fix(f)$, we have ${\lim_{n \to \infty}I(f,x)(n) = x}$.
\item $I$ is called \demph{avoidant} if for all $f \in \Ne(K)$ and $x \in K$, we have the implication
\[\left(\exists n. f\left(I(f,x)(n)\right) = I(f,x)(n) \right) \Rightarrow f(x) = x. \]
\item $I$ is called \demph{simple} if it is of the form
\begin{align*}
&I(f,x)(0) = x  \\
&I(f,x)(n + 1) = \sum_{k = 0}^n \alpha_k^nI(f,x)(k) + \sum_{j,k = 0}^n \beta_{j,k}^n f^{(j)}(I(f,x)(k)),
\end{align*}
with $\alpha_k^n \geq 0$ and $\beta_{k,j}^n \geq 0$ for all $k,n,j$.
\end{enumerate}
\end{Definition}

\noindent The notion of projectiveness is well-defined thanks to Theorem \ref{thm: projection theorem} and Proposition \ref{prop: fixed point sets of nonexpansive mappings are convex}. Any projective fixed point iteration is clearly retractive. Note that Halpern's iteration (where by convention we always choose the anchor point to be equal to the starting point) is projective and simple and that the Krasnoselski-Mann iteration is simple, retractive, and avoidant.

\begin{proposition}\label{prop: simple and retractive in [0,1] is projective}\hfill
\begin{enumerate}[label=(\roman*)]
\item Let $I\colon\Ne([0,1])\times[0,1] \to [0,1]^{\N}$ be a simple and retractive fixed point iteration. Then $I$ is projective when restricted to the set of all monotonically increasing functions.
\item Let $E$ be a uniformly convex real Banach space, and let $K \subseteq E$ be nonempty, convex, closed, and bounded. Let $I\colon \Ne(K)\times K \to K^{\N}$ be an avoidant fixed point iteration. Then for all nonexpansive $f\colon K\to K$ and $x \notin \Fix(f)$, $\lim_{n\to \infty} I(f,x)(n)$ is a point on the boundary of $\Fix(f)$.
\end{enumerate}
\end{proposition}
\begin{proof}\hfill
\begin{enumerate}[label=(\roman*)]
\item By induction one easily verifies that if $x \leq y$, then $I(f,x)(n) \leq I(f,y)(n)$ for all monotonically increasing $f\colon[0,1]\to[0,1]$. It follows that \[{\lim_{n\to\infty}I(f,x)(n) \leq \lim_{n\to\infty}I(f,y)(n)}.\] By Proposition \ref{prop: fixed point sets of nonexpansive mappings are convex}, the set $\Fix(f)$ is an interval of the form $[a,b]$, possibly with $a = b$. If $x\leq a$, then, since the iteration is retractive, 
\[\lim_{n \to \infty} I(f,x)(n) \leq \lim_{n \to\infty} I(f,a)(n) = a,\] so $\lim_{n \to \infty} I(f,x)(n) = a$. An analogous argument applies if $x \geq b$. It follows that the mapping $\lambda x.\lim_{n\to\infty} I(f,x)(n)$ is the metric projection onto $[a,b]$, i.e. the iteration is projective.
\item Is trivial.\qedhere
\end{enumerate}
\end{proof}

\noindent The following proposition establishes the more obvious relationships between the problems introduced in Definition \ref{Def: computational problems}. 

\begin{proposition}\label{prop: obvious Weihrauch reductions}
Let $E$ be a uniformly convex computable Banach space, let $K \subseteq E$ be nonempty, co-semi-decidable, computably overt, bounded, and convex. Then
\[\BGK_K \leq_W \Proj_K.\]
If $I\colon \Ne(K) \times K \to K^{\N}$ is a computable fixed point iteration, then
\[\BGK_K \leq_W \lim(I) \leq_W \Conv_I \leq_W \lim.\]
If $I$ is projective, then 
\[\Proj_K \leq_W \lim(I).\eqno{\qEd}\]
\end{proposition}

\noindent Next we prove a general upper bound on the Weihrauch degree of $\Proj_K$ (and thus of $\BGK_K$). We need two lemmas which constitute the main steps in Goebel's proof \cite{goebelBGK} of the Browder-Göhde-Kirk theorem (see also the proof of \cite[Theorem IV.7.13]{Werner}).

\begin{lemma}\label{lem: goebels lemma}
Let $E$ be a uniformly convex Banach space, let $K\subseteq E$ be nonempty, convex, closed, and bounded. Then there exists a function $\varphi\colon(0,1)\to(0,1)$ with ${\lim_{\varepsilon \to 0} \varphi(\varepsilon) = 0}$, such that for every nonexpansive mapping $f\colon K \to K$ and all $x, y \in K$ we have the implication
\[\left(\norm{x - f(x)} < \varepsilon \land \norm{y - f(y)} < \varepsilon\right) \rightarrow \bignorm{\frac{x + y}{2} - f\left(\frac{x + y}{2}\right)} < \varphi(\varepsilon) .\eqno{\qEd}\]
\end{lemma}

\noindent Actually, $\varphi(\varepsilon)$ is given by a very simple term involving $\varepsilon$ and the modulus of uniform convexity of $E$, but we do not need this fact here.

\begin{lemma}\label{lem: infimum lemma}
Let $E$ be a uniformly convex Banach space, let $K \subseteq E$ be nonempty, convex, closed, and bounded, and let $f\colon K \to K$ be nonexpansive. Let $A \subseteq K$ be nonempty, closed, and convex. Then $A$ intersects the fixed point set of $f$ if and only if
\[\inf\{\norm{f(x) - x} \;\big|\; x \in A\} = 0.\]
\end{lemma}
\begin{proof}
Clearly, if $A$ intersects the fixed point set of $f$, then $\inf\{\norm{f(x) - x} \;\big|\; x \in A\} = 0$. On the other hand, suppose that $\inf\{\norm{f(x) - x} \;\big|\; x \in A\} = 0$. Let
\[\mu(s) = \inf\{\norm{f(x) - x} \;\big|\; x \in A, \norm{x} \leq s \} \]
and
\[r = \inf\{s > 0 \; \big|\; \mu(s) = 0\}.\]
Since $K$ is bounded and $\inf\{\norm{f(x) - x} \;\big|\; x \in A\} = 0$, $r$ is a well-defined real number. Let $(x_n)_n$ be a sequence in $A$ with
\[\lim_{n \to \infty} \norm{f(x_n) - x_n} = 0 \]
and
\[\lim_{n \to \infty}\norm{x_n} = r. \]
We will show that $(x_n)_n$ is a Cauchy sequence. It then follows that $(x_n)_n$ converges to a fixed point, which proves the claim. Suppose that $(x_n)_n$ is not a Cauchy sequence. Then $r > 0$ and there exists $\varepsilon > 0$ and a subsequence $(y_n)_n$ of $(x_n)_n$ such that $\norm{y_{n + 1} - y_n} \geq \varepsilon$ for all sufficiently large $n$. Let $2r \geq s > r$ be such that
\[\left(1 - \eta_E\left(\frac{\varepsilon}{2r}\right)\right)s < r.\]
For $n$ sufficiently large we have $\norm{y_n} \leq s$, so that we have the inequalities
\[\bignorm{\frac{y_n}{s}} \leq 1,\; \bignorm{\frac{y_{n + 1}}{s}} \leq 1,\; \bignorm{\frac{y_n}{s} - \frac{y_{n + 1}}{s}} \geq \frac{\varepsilon}{s}. \]
Applying uniform convexity, we obtain 
\[\bignorm{\frac{y_n + y_{n + 1}}{2}} \leq s\left(1 - \eta_E\left(\frac{\varepsilon}{s}\right)\right), \]
which yields (using that without loss of generality, $\eta_E$ is monotonically increasing)
\[\bignorm{\frac{y_n + y_{n + 1}}{2}} \leq s\left(1 - \eta_E\left(\frac{\varepsilon}{2r}\right)\right) < r. \]
Now, by Lemma \ref{lem: goebels lemma} we have
\[\lim_{n \to \infty}\bignorm{\frac{y_n + y_{n + 1}}{2} - f\left(\frac{y_n + y_{n + 1}}{2}\right)} = 0. \]
This contradicts the minimality of $r$. Hence, $(x_n)_n$ is a Cauchy sequence.
\end{proof}

\begin{proposition}\label{prop: initial upper bound on Proj}
Let $E$ be a uniformly convex computable Banach space, let $K \subseteq E$ be nonempty, co-semi-decidable, computably overt, bounded, and convex. Then we have
\[\Proj_K \leq_W \lim \circ \lim.\]
If $K$ is computably compact or $E$ is a Hilbert space, then
\[\Proj_K \leq_W \lim. \]
\end{proposition}
\begin{proof}
We are given as input a nonexpansive function $f\colon K \to K$ and a point $x \in K$ and want to obtain the point $p = P_{\Fix(f)}(x) \in K$. In the case where $E$ is a Hilbert space, we can use Halpern's iteration (\ref{thm: Halpern's theorem}) to obtain a computable sequence converging to $p$ and apply $\lim$ to obtain $p$ itself. 

In the case where $K$ is computably compact, we can compute $\Fix(f)$ as an element of $\A_{\dist_<}(K)$ (see the discussion after Definition \ref{Def: distance representations}). In particular we can compute a sequence $(d(x_n,\Fix(f))_{n \in \N}$, where $(x_n)_n$ is a computable dense sequence in $K$, as an element of $\R_{<}^{\N}$. Using the standard identification of $\N^{\N}$ with $\left(\N^{\N}\right)^{\N}$, it is easy to see that we can use a single instance of $\lim$ to obtain countably many instances of $\lim$ in parallel (cf.~also e.g.~\cite{ClosedChoice} or \cite{WeihrauchDegrees}). Using $\lim$, we can hence compute the sequence $(d(x_n,\Fix(f))_{n \in \N}$ as an element of $\R^{\N}$, which allows us to compute $\lambda y.d(y,\Fix(f))$ as an element of $\Co(K,\R)$, since we have 
\[\left|d(y,\Fix(f)) - d(z, \Fix(f)) \right| \leq d(y,z)\] for all $y ,z \in K$. 

Independently, we can use the same instance of $\lim$ to obtain a modulus of uniform convexity $\eta_E$ for $E$: we may put
\[\eta_E(\varepsilon) = \inf\Big\{1 - \bignorm{\frac{x + y}{2}} \;\big|\; x,y \in B_E, \norm{x - y} \geq \varepsilon \Big\}, \]
and if $(x_n)_n$ is a computable dense sequence in $B_E$ (we may choose e.g. the sequence of rational points contained in the open unit ball) we have
\[\eta_E(\varepsilon) = \inf\Big\{1 - \bignorm{\frac{x_i + x_j}{2}} \;\big|\; i,j \in \N, \norm{x_i - x_j} > \varepsilon \Big\},\]
which is clearly limit-computable\footnote{We call a multimapping \emph{limit-computable} if its Weihrauch degree is below $\lim$.} in $\varepsilon$. This allows us to compute the restriction of $\eta_E$ to the rational numbers using countably many applications of $\lim$. In particular we can limit-compute a function $\mu\colon \N \to \N$ satisfying $2^{-\mu(n)} \leq \eta_E(2^{-n})$. Using the distance function of $\Fix(f)$ and the function $\mu$, we apply Corollary \ref{cor: computability of projection on subset} to obtain the projection of $x$ onto $\Fix(f)$.

In the general case, we cannot a-priori compute $\Fix(f)$ as an element of $\A_{\dist_<}(K)$, because of Proposition \ref{prop: equivalence lower semi-located and computably closed implies locally compact}. We can however use $\lim$ to obtain $\Fix(f)$ as an element of $\A_{\dist_<}(K)$: since $K$ is computably overt, we can list all rational closed balls $\clos{B}(a,r)$ for which the open ball $B(a,r)$ intersects $K$. Given such a rational closed ball $\clos{B}(a,r)$ in $E$, we can compute a dense sequence in $\clos{B}(a,r) \cap K$: choose a computable dense sequence $(x_n)_n$ in $K$ and filter out those points $x_n$ which satisfy $d(x_n,a) < r$. Using the convexity of $K$ it is easy to see that the resulting sequence is dense in $\clos{B}(a,r) \cap K$. This allows us to limit-compute
\[\inf\big\{\norm{f(x) - x} \;\big|\; x \in \clos{B}(a,r) \cap K\big\}.\]
Again, we can do this for all suitable closed rational balls in parallel. We can then enumerate those balls $\clos{B}(a,r)$ satisfying $\inf\{\norm{f(x) - x} \;\big|\; x \in \clos{B}(a,r) \cap K\} > 0$, which by Lemma \ref{lem: infimum lemma} is equivalent to $\clos{B}(a,r) \cap \Fix(f) = \emptyset$. This allows us to compute the distance function to $\Fix(f)$ from below (cf.~\cite[Theorem 3.9 (1)]{BrattkaPresser} or the proof of \cite[Lemma 5.1.7]{Weih}). Now we apply the limit-computable method used in the compact case above to obtain the projection. Since the Weihrauch degree of the composition of two limit-computable mappings is below $\lim\circ\lim$ (see e.g. Fact 8.2 in \cite{BolzanoWeierstrass}), the result follows.
\end{proof}

Note that if $E$ is finite dimensional, then $K$ is always computably compact, so that the stronger upper bound of Proposition \ref{prop: initial upper bound on Proj} applies. We will show in Theorem \ref{thm: upper bound on Proj in smooth spaces} that $\lim$ is an upper bound on $\Proj_K$ in all uniformly convex and uniformly smooth computable Banach spaces, and thus for instance in all $L^p$-spaces with $1 < p < \infty$.

We now begin a discussion on the computability of fixed points and the Weihrauch degree of the Browder-G\"ohde-Kirk theorem. Proposition \ref{prop: fixed point sets of nonexpansive mappings are convex} yields an immediate upper bound for the Weihrauch degree of $\BGK_K$.

\begin{proposition}
Let $E$ be a uniformly convex computable Banach space. Let $K \subseteq E$ be nonempty, co-semi-decidable, computably overt, bounded, and convex. Then we have $\BGK_K \leq_W \XC_K$.\qed
\end{proposition}
In finite dimension, this already implies that $\BGK$ is always strictly weaker than $\WKL$. In fact it is non-uniformly computable thanks to the following result due to Le Roux and Ziegler.
\begin{theorem}[\cite{ZR}]\label{thm: convex sets have computable points}
Let $E$ be a finite dimensional computable Banach space. Let $K \subseteq E$ be nonempty, co-semi-decidable, and convex. Then $K$ contains a computable point.
\end{theorem}
\begin{proof}[Proof sketch.]
We may assume that $K$ is compact, since we can always intersect $K$ with a sufficiently large closed ball. We may also assume that $E$ is represented by $\rho^d$ for some $d \in \N$. We proceed by induction on $\dim E$. If $\dim E = 1$, then $K$ is either a singleton and hence computable, or it is an interval and hence contains a rational point. If $\dim E = d$, then the projection of $K$ onto the $x$-axis is still nonempty, co-semi-decidable, and convex, and contains a computable point $x$ by induction hypothesis. Now, the intersection of $\{x\}\times\R^{d - 1}$ with $K$ is nonempty, convex, and co-semi-decidable, of dimension strictly smaller than $d$. Again by induction hypothesis, the intersection, and in particular $K$, contains a computable point.
\end{proof}
The above theorem even shows that $K$ has a dense subset of computable points, since the intersection of $K$ with a small rational ball is again co-semi-decidable and compact. A more uniform version, using Weihrauch degrees, has been given in \cite{PaulyLeRoux}.

\begin{corollary}\label{cor: fixed points in finite dimension are computable}
Let $E$ be a finite dimensional, strictly convex computable Banach space. Let $K \subseteq E$ be nonempty, computably overt, co-semi-decidable, bounded, and convex and let $f\colon K \to K$ be computable and nonexpansive. Then $f$ has a computable fixed point.\qed
\end{corollary}

Corollary \ref{cor: fixed points in finite dimension are computable} in particular shows that the Browder-G\"ohde-Kirk theorem in finite dimension is strictly more effective than the (in this case more general) Brouwer fixed point theorem: a construction due to Orevkov \cite{Orevkov} and Baigger \cite{Baigger} shows that there exists a computable function on the unit square in $\R^2$ without computable fixed points, while in every finite dimension some fixed points whose existence is guaranteed by the Browder-G\"ohde-Kirk theorem are computable. Note that Corollary \ref{cor: fixed points in finite dimension are computable} really only uses the fact that $f$ is computable and its fixed point set is convex. Thus, Theorem \ref{thm: convex sets have computable points} presents a fairly general non-uniform computability result: if a computable equation on a finite dimensional space has a convex set of solutions, then it has a computable solution. A nontrivial application is based on the following result\footnote{This application was pointed out to the author by Ulrich Kohlenbach.}. A self-map $f\colon K \to K$ of a nonempty subset $K$ of a real Hilbert space $H$ is called \emph{pseudocontractive} if it satisfies \[\product{f(x) - f(y)}{x - y} \leq \norm{x - y}^2 \]
for all $x,y \in K$.
\begin{theorem}[\cite{ZhangCheng}]
Let $H$ be a real Hilbert space, let $K \subseteq H$ be nonempty, closed and convex, and let $f\colon K \to K$ be pseudocontractive. Then $\Fix(f)$ is closed and convex.\qed
\end{theorem}

\begin{corollary}\label{cor: computability of fixed points of pseudocontractions}
Let $H$ be a finite dimensional computable Hilbert space. Let $K \subseteq H$ be nonempty, computably overt, co-semi-decidable, bounded, and convex, and let $f\colon K\to K$ be computable and pseudocontractive. Then $f$ has a computable fixed point.\qed
\end{corollary}

Corollary \ref{cor: computability of fixed points of pseudocontractions} is strictly more general than the Euclidean version of Corollary \ref{cor: fixed points in finite dimension are computable}, as the following example, due to \cite{ChidumeMutangadura}, shows. 

\begin{proposition}\label{prop: computable Lipschitz pseudocontraction which is not nonexpansive}
There exists a computable Lipschitz-continuous pseudocontractive mapping on the unit ball of Euclidean $\R^2$, which is not nonexpansive.
\end{proposition}
\begin{proof}
For $x = (x_1,x_2)$, we put $x^{\perp} = (-x_2,x_1)$. Let
\[f(x) = \begin{cases} x + x^{\perp} &\text{ if }\norm{x} \leq \tfrac{1}{2},\\ \tfrac{x}{\norm{x}} - x + x^{\perp} &\text{ if }\norm{x} \geq \tfrac{1}{2}.\end{cases} \]
It is easy to see that $f$ is computable and Lipschitz-continuous with Lipschitz constant $5$. The proof of pseudocontractiveness is somewhat technical and can be found in \cite{ChidumeMutangadura}.
\end{proof}

The function from Proposition \ref{prop: computable Lipschitz pseudocontraction which is not nonexpansive} actually provides an example of a pseudocontractive mapping with a unique fixed point, for which the Krasnoselski-Mann iteration, which is guaranteed to converge for nonexpansive mappings, fails to converge. This is proved in \cite{ChidumeMutangadura}.

In infinite dimension, there exist computable firmly nonexpansive mappings without computable fixed points, already on compact sets. A mapping $f\colon K \to K$ defined on a nonempty subset $K$ of a real Hilbert space $H$ is called \emph{firmly nonexpansive} if it satisfies
\[\norm{f(x) - f(y)}^2 \leq \product{x - y}{f(x) - f(y)} \]
for all $x, y \in K$. Clearly, every firmly nonexpansive mapping is nonexpansive. It is not difficult to see that a mapping is firmly nonexpansive if and only if it is of the form ${f(x) = \tfrac{1}{2}(x + g(x))}$, where $g$ is a nonexpansive mapping. Let \[\mathcal{H} = \{x \in \ell^2 \;|\; 0 \leq x(n) \leq 2^{-n} \;\text{for all }n \in \N\}\] denote the Hilbert cube in $\ell^2$ (represented by $\delta_{\ell^2}\big|^{\mathcal{H}}$).

\begin{theorem}\label{thm: nonexpansive mapping on Hilbert cube without computable fixed points}
There exists a computable firmly nonexpansive mapping $f\colon \mathcal{H} \to \mathcal{H}$ without computable fixed points.
\end{theorem}
\begin{proof}
Put $g_n(x) = (1-2^{-n})x$ and $h_n(x) = 2^{-n} + (1 - 2^{-n})x$. Then $(g_n)_n$ and $(h_n)_n$ are computable sequences of nonexpansive self-maps of $[0,1]$, satisfying $|g_n(x) - x| \leq 2^{-n}$, $|h_n(x) - x| \leq 2^{-n}$ for all $x \in [0,1]$, $\Fix(g_n) = \{0\}$, and $\Fix(h_n) = \{1\}$.

Let $A,B \subseteq \N$ be two disjoint, recursively enumerable, and recursively inseparable sets. Let $\alpha$ be the G\"odel number of an algorithm with halting set $A$ and $\beta$ be the G\"odel number of an algorithm with halting set $B$. Consider the sequence of functions $(f_n)_n$ with
\[f_n(x) = \begin{cases}g_i(x) &\text{ if }\alpha\text{ halts on input }n\text{ within }i\text{ steps,}\\
					    h_i(x) &\text{ if }\beta\text{ halts on input }n\text{ within }i\text{ steps,} \\
					    	x &\text{ if both }\alpha\text{ and }\beta\text{ diverge on input }n. \end{cases} \]
Note that since $A$ and $B$ are disjoint, both $\alpha$ and $\beta$ cannot halt on the same input, so $f_n$ is well-defined. The sequence $(f_n)_n$ is a computable sequence: in order to compute $f_n(x)$ up to error $2^{-m}$, we simulate $\alpha$ and $\beta$ simultaneously on input $n$ for $m$ steps. If $\alpha$ (respectively $\beta$) halts within $k \leq m$ steps, we output $g_k(x)$ (respectively $h_k(x)$) up to error $2^{-m}$. If neither $\alpha$ nor $\beta$ halt after $m$ steps, we may output $x$ as an approximation, since $|g_{m + k}(x) - x| \leq 2^{-m-k}$ and $|h_{m + k}(x - x)| \leq 2^{-m-k}$ for all $k \geq 0$. Now, suppose there exists a computable sequence $(x_n)_n$ with $f_n(x_n) = x_n$. In order to arrive at a contradiction, we use $(x_n)_n$ to construct a computable set $S\subseteq \N$ separating $A$ and $B$. Membership for $S$ is decided as follows: for a given $n \in \N$, run the tests $x_n > 0$ and $x_n < 1$ simultaneously. At least one of the tests has to succeed. If the first test to succeed is $x_n > 0$, we decide that $n \notin S$. If the first test to succeed is $x_n < 1$, we decide that $n \in S$. Note that in the case where both $x_n > 0$ and $x_n < 1$, the outcome of the decision procedure may depend on the Cauchy sequence of dyadic rational numbers representing $x_n$. We claim that $S \supseteq A$ and $\comp{S} \supseteq B$. If $n \in A$, then $\alpha$ halts on input $n$ after $i \in \N$ steps. So $f_n = g_i$, and thus $x_n = 0$. The test $x_n > 0$ will hence fail, while the test $x_n < 1$ will succeed and thus $n \in S$. If $n \in B$, then $\beta$ halts on input $n$ after $i \in \N$ steps, so $f_n = h_i$ and thus $x_n = 1$. It follows that $n \notin S$. So $S$ separates $A$ and $B$. Contradiction.
Now define a nonexpansive mapping $g\colon\mathcal{H}\to\mathcal{H}$ via $g(x)(n) = 2^{-n}f_n(2^nx(n))$. Then $g$ is computable, for in order to compute an approximation to $g(x)$ in $\ell^2$ up to error $2^{-n}$, it suffices to compute the real numbers $g(x)(0),\dots,g(x)(n + 1)$ up to error $2^{-2n - 1}/(n + 2)$. Any fixed point for $g$ can be used to compute a sequence of fixed points for $(f_n)_n$. In particular, $g$ has no computable fixed points. In order to obtain a firmly nonexpansive mapping $f$ we put $f = \tfrac{1}{2}(\id + g)$.
\end{proof}

Theorem \ref{thm: nonexpansive mapping on Hilbert cube without computable fixed points} in particular shows that fixed points of (firmly) nonexpansive mappings are not computable relative to any discrete advice. Let us now consider the computability of rates of convergence of certain fixed point iterations. While in infinite dimension, the non-uniform uncomputability of fixed points in particular implies that there exist computable mappings such that no computable fixed point iteration has a computable rate of convergence for any computable starting point, and Theorem \ref{uniform uncomputability of fixed points} tells us that there is no general algorithm for obtaining rates of convergence uniformly in the input function and in the starting point, it might still be the case (at least in finite dimension) that there exists a computable fixed point iteration $I$ such that for every computable nonexpansive function $f$ there exists a computable starting point $x_0 \notin \Fix(f)$ such that the sequence $I(f,x_0)$ has a computable rate of convergence. This could still be practically relevant, since in a given practical scenario one might be able to exploit additional information on the input function in order to choose the starting point of the iteration in such a way that the rate of convergence becomes computable. We will however see that this fails to be the case for a large class of fixed point iterations, already on the compact unit interval.

We prove a special case of our main result (Theorem \ref{main result}) where the underlying set is the the compact unit interval $[0,1]$. Our theorem uniformly characterises the fixed point sets of computable nonexpansive self maps of $[0,1]$. We first recall an elementary fact (cf.~\cite{Weih}). 

\begin{proposition}\label{prop: co-semi-decidable intervals and their endpoints} 
Let $\mathcal{I} = \{(a,b) \in (\R_<\times \R_>)\;|\; a \leq b\}$. Then the mapping
\[\A^{\conv}(\R)\setminus\{\emptyset\} \to \mathcal{I},\; [a,b] \mapsto (a,b) \]
and its inverse
\[\mathcal{I} \to \A^{\conv}(\R)\setminus\{\emptyset\},\; (a,b) \mapsto [a,b] \]
are computable. \qed
\end{proposition}

\begin{theorem}\label{thm: main theorem on [0,1]}
\hfill
\begin{enumerate}[label=(\roman*)]
\item The mapping
\[\Fix\colon \Ne([0,1]) \to \A^{\conv}([0,1])\setminus\{\emptyset\}, \; f \mapsto \Fix(f) \]
is computable.
\item And so is its multivalued inverse
\[\Fix^{-1}\colon \A^{\conv}([0,1])\setminus\{\emptyset\} \rightrightarrows \Ne([0,1]).\] 
\end{enumerate}
\end{theorem}
\begin{proof}
The first claim immediately follows from Proposition \ref{prop: fixed point sets of nonexpansive mappings are convex} and the well-known result that the set of zeroes of a continuous mapping $f$ is co-semi-decidable in $f$.

Let us now prove the second claim. Suppose we are given a nonempty, closed interval $[a,b] \in \A^{\conv}([0,1])$.  By Proposition \ref{prop: co-semi-decidable intervals and their endpoints} we can compute a monotonically increasing list $(a_n)_n$ of rational numbers converging from below to $a$, and a monotonically decreasing list $(b_n)_n$ of rational numbers converging from above to $b$. We may assume without loss of generality that $a_n \geq 0$ and $b_n \leq 1$ for all $n \in \N$.

From $(a_n)$ and $(b_n)$ we can compute a sequence $(f_n)_n$ of nonexpansive functions via
\[f_n(x) = \begin{cases}a_n &\text{if }x \leq a_n\text{,}\\ x &\text{if }a_n \leq x \leq b_n\text{,} \\ b_n &\text{if }x\geq b_n\text{.}\end{cases}\]
Finally, we compute
\[f(x) = \sum_{n \in \N} 2^{-n-1}f_n(x).\]
Then $f$ is nonexpansive, and maps $[0,1]$ into $[0,1]$. Let us now show that $\Fix(f) = [a,b]$. If $x \in [a,b]$ then $f_n(x) = x$ for all $n \in \N$, so ${f(x) = \sum_{n \in \N} 2^{-n - 1}x = x}$. Suppose now without loss of generality that $x < a$. Then $f_n(x) \geq x$ for all $n \in \N$, and there exists $m\in\N$ such that $x < a_m$ and hence $f_m(x) > x$. It follows that 
\[f(x) = \sum_{n \in \N} 2^{-n - 1}f_n(x) > x,\] and hence $x\notin\Fix(f)$. An analogous argument applies if $x > b$. We may hence put ${\Fix^{-1}([a,b]) = f}$.
\end{proof}

From the proof we obtain the following non-uniform corollary, which is slightly stronger than the non-uniform version of Theorem \ref{thm: main theorem on [0,1]}.

\begin{corollary}\label{cor: construction of firmly nonexpansive function with given fixed point set}
Let $[a,b] \subseteq [0,1]$ be a co-semi-decidable interval. Then there exists a monotonically increasing, firmly nonexpansive, computable function $f\colon [0,1] \to [0,1]$ with $\Fix(f) = [a,b]$.
\end{corollary}
\begin{proof}
The algorithm we use in the proof of Theorem \ref{thm: main theorem on [0,1]} to compute $\Fix^{-1}$ maps any nonempty co-semi-decidable interval to a monotonically increasing, nonexpansive function ${g\colon [0,1] \to [0,1]}$ with $\Fix(g) = [a,b]$. In order to obtain a firmly nonexpansive function $f$, we put $f(x) = \tfrac{1}{2}(x + g(x))$.
\end{proof}

Since by Proposition \ref{prop: co-semi-decidable intervals and their endpoints} any left-r.e.\ number can be the left endpoint of a co-semi-decidable interval, we obtain the announced result together with Proposition \ref{prop: simple and retractive in [0,1] is projective}.
\begin{corollary}\label{cor: hardness of rates of convergence}
Let $I\colon \Ne([0,1])\times [0,1] \to [0,1]^{\N}$ be an either projective, or simple and retractive, or avoidant computable fixed point iteration. Let $\varepsilon > 0$. There exists a computable, firmly nonexpansive function $f\colon [0,1] \to [0,1]$ with $\operatorname{diam}(\Fix(f)) < \varepsilon$ such that for no computable $x \notin \Fix(f)$, the sequence $I(f,x)$ has a computable rate of convergence.
\end{corollary}
\begin{proof}
Let $a \in (0,1)$ be an uncomputable left-r.e.\ number and $b \in (0,1)$ be an uncomputable right-r.e.\ number with $|a - b| < \varepsilon$. Then the closed interval $[a,b] \subseteq [0,1]$ is co-semi-decidable by Proposition \ref{prop: co-semi-decidable intervals and their endpoints}. Using Corollary \ref{cor: construction of firmly nonexpansive function with given fixed point set} we obtain a monotonically increasing firmly nonexpansive function $f\colon[0,1]\to[0,1]$ with $\Fix(f) = [a,b]$. If $x \notin \Fix(f)$, then by Proposition \ref{prop: simple and retractive in [0,1] is projective}, we have $\lim_{n \to \infty}I(f,x)(n) \in \{a,b\}$. In particular, $\lim_{n \to \infty}I(f,x)(n)$ is uncomputable. Since $x$ is computable, the sequence $(I(f,x)(n))_n$ is a computable sequence of real numbers, so if it had a computable rate of convergence, its limit would be computable. This proves the claim.
\end{proof}

Both the Halpern iteration and the Krasnoselski-Mann iteration are simple and retractive, so Corollary \ref{cor: hardness of rates of convergence} applies to them. Using Weihrauch degrees, we can state our present results more uniformly.
\begin{proposition}
We have $\Proj_{[0,1]} \equiv_W \lim$. If $I\colon \Ne([0,1])\times [0,1] \to [0,1]^{\N}$ is an either projective, or simple and retractive, or avoidant computable fixed point iteration, then $\lim(I) \equiv_W \Conv_I \equiv_W \lim$.
\end{proposition}
\begin{proof}
It is well known that the identity $\id\colon \R_{<} \to \R$ is Weihrauch-equivalent to $\lim$, even when restricted to the unit interval. Given $a \in \R_{<} \cap [0,1]$, we can compute the interval $[a,1] \in \A^{\conv}([0,1])$, and hence, by Theorem \ref{thm: main theorem on [0,1]}, construct a nonexpansive function $f\colon [0,1] \to [0,1]$ with $\Fix(f) = [a,1]$. Now, $\Proj_{[0,1]}(f,0) = a$, so $\lim \leq_W \Proj_{[0,1]}$. Together with Proposition \ref{prop: obvious Weihrauch reductions} we obtain $\Proj_{[0,1]} \equiv_W \lim$. If $I$ is a projective, simple and retractive, or avoidant, computable fixed point iteration, then by Proposition \ref{prop: simple and retractive in [0,1] is projective} we obtain $\Proj_{[0,1]} \leq_W \lim(I)$, and thus $\lim \leq_W \lim(I) \leq_W \Conv_I \leq_W \lim$, i.e. $\lim(I) \equiv_W \Conv_I \equiv_W \lim$.
\end{proof}

\begin{proposition}\label{prop: BGK in [0,1]}
$\BGK_{[0,1]}\equiv_w \XC_{[0,1]} \equiv_W \IVT \equiv_W \BFT_1$.\qed
\end{proposition}

The equivalence $\BGK_{[0,1]}\equiv_W \XC_{[0,1]}$ follows immediately from Theorem \ref{thm: main theorem on [0,1]}, the equivalence $\XC_{[0,1]} \equiv_W \IVT \equiv_W \BFT_1$ was already stated in Fact \ref{fact on Weihrauch degrees}. So far, there seems to be a significant discrepancy between the computational content of the existence result $\BGK$ and the ``constructive'' theorems by Wittmann and Krasnoselski. We will see in Section \ref{Section: 5} that this discrepancy disappears on non-compact domains in infinite-dimensional Hilbert space, where the Browder-Göhde-Kirk theorem is Weihrauch equivalent to $\lim$, and hence to Wittmann's theorem. 
\section{Weak Topologies}
In order to be able to prove our main result in full generality, we have to introduce an admissible representation for the weak topology on a reflexive Banach space $E$. Such a representation has first been introduced by Brattka and Schr\"oder \cite{BrattkaSchroeder}. We denote the continuous dual of a normed space $E$ by $E'$ and define the mapping
\[\product{\cdot}{\cdot}\colon E\times E' \to \R,\; (x,x') \mapsto x'(x).\]

\begin{Definition}\label{Def: admissible representation for weak* topology}
Let $E$ be a computable Banach space. The represented space $E'_w$ is the space $E'$, represented via the co-restriction of $[\delta_E\to\rho]$ to all continuous linear functionals.
\end{Definition}

\begin{theorem}[\cite{BrattkaSchroeder}]\label{thm: admissibility of delta E'_w}
The representation $[\delta_E\to\rho]\Big|^{E'}$ is admissible with respect to the weak* topology on $E'$.\qed
\end{theorem}

Since the points of $E$, viewed as functionals on $E'$, separate the points of $E'$, the weak* topology on $E'$ is Hausdorff. By Theorem \ref{thm: admissibility of delta E'_w}, the space $E'_w$ is then a Hausdorff represented topological space, so that the space $\K(E'_w)$ is well-defined and coincides extensionally with the set of all weak* compact subsets of $E'$. Note that this crucially relies on the separability of $E$, since since weak* sequential compactness and weak* compactness need not coincide on duals of inseparable spaces. The fact that they do coincide in the separable case also follows from the well-known fact that the weak* topology on the dual space of a separable Banach space is metrisable on the unit ball (cf.~Theorem \ref{thm: unit ball in weak*-topology as a computable metric space} below). Also note that the weak* topology on $E'$ is in general not sequential (i.e. there exist sequentially weak* closed sets which are not weak* closed). Consequently, the spaces $\O(E_w')$ and $\A(E_w')$ do not coincide with the hyperspaces of weak* open and weak* closed sets respectively, but with the hyperspaces of weak* sequentially open and weak* sequentially closed sets. If $A \in \A(E'_w)$, we write $A_w$ for the represented space $(A,\delta_{E'_w}\big|^A)$ to emphasize the underlying representation.

If $E$ is a reflexive real Banach space with computable dual $E'$, we obtain a canonical representation for $E$ with respect to the weak topology, by identifying $E$ with $E''$ and putting $E_w = (E')'_w$, i.e. $E_w$ is the represented space $E''$ with representation $[\delta_{E'}\to\rho]\big|^{E''}$ (using that in this case the weak* topology on $E''$ coincides with the weak topology on $E$). Again, the space $\K(E_w)$ and the space of weakly compact subsets of $E$ coincide extensionally\footnote{Thus, Theorem \ref{thm: admissibility of delta E'_w} together with Proposition 3.3.2 (3) in \cite{SchroederPhD} provide an interesting proof of the separable \emph{Eberlein-{\v S}mulian theorem} as well as its analogue for the weak* topology in duals of separable spaces.}, but the caveat on $\O(E'_W)$ and $\A(E'_w)$ also applies to $\O(E_w)$ and $\A(E_w)$. As in the case of $E'_w$, if $A \in \A(E_w)$, we write $A_w$ for the represented space $(A,\delta_{E_w})$. We will often use the adjective ``weak'' when referring to elements in hyperspaces constructed from $E_w$. For instance, we may call the computable points of $\A(E_w)$ ``weakly co-semi-decidable'' and the computable points of $\A(E'_w)$ ``weak* co-semi-decidable'' etc.

\begin{Remark}
Note that in the definition of $E_w$ we only require $E'$, but not $E$ itself, to be a computable Banach space. By definition, the mapping $\product{\cdot}{\cdot}$ is $(\delta_{E_w}\times\delta_{E'},\rho)$-computable. If both $E$ and $E'$ are computable Banach spaces, it is natural to require that the mapping $\product{\cdot}{\cdot}$ be $(\delta_{E}\times\delta_{E'},\rho)$-computable, so that $\id\colon E_w \to E$ becomes computable (see for instance Theorem \ref{thm: unit ball in weak*-topology as a computable metric space} (iii), Corollary \ref{cor: Banach space is isomorphic to subspace of function space} and Propositions \ref{prop: strongly overt implies weakly overt}, \ref{prop: weak and strong topology equivalent on compact sets} and \ref{prop: computably weakly closed uniformly implies lower semi-located} below). This is for instance the case for the spaces $L^p([0,1])$ with $1 < p < \infty$, since we have $(L^p([0,1]))' = L^q([0,1])$, where $\tfrac{1}{p} + \tfrac{1}{q} = 1$, and if $f \in L^p([0,1])$ and $g \in L^q([0,1])$, then $\product{f}{g}$ is given by the effective formula
\[\product{f}{g} = \int_0^1 f(x) g(x) \operatorname{dx}.\]
\end{Remark}

Next, we prove some basic properties of the space $\K(E'_w)$. We will need a few effective counterparts to classical results from functional analysis. The first is an effective version of the separable Banach-Alaoglou theorem, which was proved by Brattka \cite{BrattkaHahnBanach}.

\begin{theorem}[Computable separable Banach-Alaoglou theorem, \cite{BrattkaHahnBanach}]\label{thm: unit ball in weak*-topology as a computable metric space}
Let $E$ be a computable Banach space. Let $B_{E'_w}$ denote the unit ball in $E'$, viewed as a subset of the represented space $E_w'$ (thus bearing the weak* topology). Then
\begin{enumerate}[label=(\roman*)]
\item $B_{E'_w} \in \K(E'_w)$.
\item More generally, let $K \subseteq E'_w$ be a co-semi-decidable subset of $E'_w$. If $K$ is bounded, then $K$ is computably weak* compact.
\item If $E'$ is a computable Banach space and $\product{\cdot}{\cdot}$ is $(\delta_{E}\times\delta_{E'},\rho)$-computable, then $B_{E'_w}$ admits the structure of a computably compact computable metric space.
\end{enumerate}
\end{theorem}
\begin{proof}
It is proved in \cite{BrattkaHahnBanach} that there exists a computable embedding $i\colon B_{E'_w} \to X$ into a computably compact computable metric space $X$, such that $i(B_{E'_w})$ is computably compact as a subset of $X$ and the partial inverse $i^{-1}\colon i(B_{E'_w}) \to B_{E'_w}$ is computable. It follows that $i(B_{E'_w})$ is a computably compact represented space, and $i$ induces a computable isomorphism between $B_{E'_w}$ and $i(B_{E'_w})$, so that $B_{E'_w}$ is computably compact as a represented space. It follows from Proposition \ref{prop: compact subspace of represented space is compact set} that $B_{E'_w}$ is a computably compact subset of $E_w'$. 
This proves the first claim.

For the second claim, observe that the mappings
\[\operatorname{mult}\colon (0,\infty)\times \A(E_w') \to \A(E'_w),\; (\alpha,A) \mapsto \alpha A = \{\alpha x\;\big|\; x \in A \}\]
and
\[\operatorname{mult}\colon (0,\infty)\times \K(E_w') \to \K(E'_w),\; (\alpha,K) \mapsto \alpha K = \{\alpha x\;\big|\; x \in K \}\]
are computable. Proposition \ref{prop: closed and compact sets, (i) closed set is compact, (ii) compact set in Hausdorff space is closed} (i) asserts that
\[f\colon \A(B_{E'_w})\to\K(B_{E'_w}),\; A \mapsto A \]
is computable. Trivially, given $A\in \A(E_w')$, such that $A \subseteq B_{E'_w}$, we can compute $A$ as a set in $\A(B_{E'_w})$. Given a bounded set $A \in \A(E'_w)$ with bound $b$, we hence obtain $A$ as an element of $\K(E'_w)$ by computing $\operatorname{mult}(b,f(\operatorname{mult}(\tfrac{1}{b},A)))$.

The third claim follows immediately from the proof of the first. We may pull back the metric $d_X$ on $X$ via $i$ to obtain a metric on $B_{E'_w}$, i.e. put $d(x,y) = d_X(i(x),i(y))$ for $x, y \in B_{E'_w}$. As the set of rational points in $B_{E'_w}$ we may choose those rational points of the computable Banach space $E'$ whose norm is strictly smaller than one. One now easily verifies that $B_{E'_w}$ is computably compact as a computable metric space, and that the Cauchy representation on $B_{E_w'}$ is computably equivalent to $[\delta_E\to\rho]\big|^{B_{E'_w}}$.
\end{proof}

As a corollary we get an effective version of a classical result in functional analysis (cf.~\cite[Korollar VIII.3.13]{Werner}) in the reflexive case. 

\begin{corollary}\label{cor: Banach space is isomorphic to subspace of function space}
Let $E$ be a reflexive computable Banach space with computable dual $E'$, such that the mapping $\product{\cdot}{\cdot}$ is $(\delta_E\times\delta_{E'},\rho)$-computable. Then $E$ is computably isometrically isomorphic to a co-semi-decidable and computably overt subspace of a function space $\Co(M)$ over a computably compact metric space $M$.
\end{corollary}
\begin{proof}
By Theorem \ref{thm: unit ball in weak*-topology as a computable metric space} (iii), $M = B_{E'_w}$ (with the weak* topology) is a computably compact computable metric space. We show that the mapping
\[i\colon E \to \Co(M), \; x \mapsto \lambda x'.(x,x')\]
is a $(\delta_E,[\delta_{B_{E'_w}}\to\rho])$-computable isometric embedding with co-semi-decidable and computably overt image. The $(\delta_E,[\delta_{B_{E'_w}}\to\rho])$-computability is obvious, and the fact that it is an isometry follows from $\norm{x} = \sup_{x' \in B_{E'}}|\product{x}{x'}|$, which in turn is an easy corollary of the Hahn-Banach theorem (cf.~e.g.~\cite[Korollar III.1.7]{Werner}). Clearly, $i(E)$ is computably overt. It is also co-semi-decidable, for if we are given a continuous function $f$ on $B_{E'_w}$, we can verify if it is nonlinear. It remains to show that its inverse $i^{-1}\colon i(E) \to E$ is computable. Given a $[\delta_{B_{E'_w}}\to\rho]$-name of $i(x) \in i(E)$ and a $\nu_E$-name of a rational point $y \in E$ we can compute a $[\delta_{B_{E'_w}}\to\rho]$-name of $i(x) - i(y)$, and hence compute 
\[\max\{|\product{x - y}{x'}|\;\big|\;x' \in B_{E'_w}\} = \norm{i(x) - i(y)} = \norm{x - y},\]
using Theorem \ref{thm: supremum on compact space}. For every $n \in \N$ we may hence search for a rational point $y_n$ in $E$ satisfying $\norm{y_n - x} < 2^{-n}$, which allows us to compute a $\delta_E$-name of $x$.
\end{proof}

\begin{Remark}\label{Rem: corollary five in Brattka Schroeder wrong}
In Corollary 5 in \cite{BrattkaSchroeder}, it is claimed that the representation $\delta^w_E$ of $E$ defined by
\[\delta^w_E(p) = x \;:\Leftrightarrow\; [\delta'_E\to\rho](p) = \iota(x),\]
where $\iota\colon E\to E''$ is the canonical embedding and $\delta'_E = [\delta_E\to\rho]\Big|^{E'}$, is admissible for the weak topology on $E$. This contradicts the proof of our Corollary \ref{cor: Banach space is isomorphic to subspace of function space}, which suggests that $\delta^w_E$ is admissible with respect to the norm topology. To convince ourselves that the claim is false, we consider the simple example of $E = \ell^2$. For the scope of this remark we will adopt the notation used in \cite{BrattkaSchroeder}. The representation $\delta^{\geq}_{\ell^2}$, where a $\delta^{\geq}_{\ell^2}$-name of $(x_n)_n \in \ell^2$ is a $\rho^\omega$-name of a sequence $(b,x_1,x_2,\dots)$ with $b \geq \norm{(x_n)_n}$, is equivalent to $\delta_{\ell^2_w}$ (which is denoted by $\delta_{\ell^2}'$ in \cite{BrattkaSchroeder}). It follows that given a $\delta^w_E(p)$-name of $x \in \ell^2$ we can compute 
$(x,x')$ for any $x' \in \ell^2$, provided that we know $x'(n)$ for all $n \in \N$ and some bound on $\norm{x'}$. In particular we can compute $x(n) = (x,e_n)$ for every $n \in \N$. Since $\delta_{\ell^2} \leq \delta_{\ell^2_w}$, a $\delta^w_{\ell^2}$-name of $x$ allows us to compute a $[\delta_{\ell^2}\to\rho]$-name of $x$. Theorem 5.1. in \cite{Brattka06computableversions} then asserts that we can compute some bound $b$ on $\norm{x}$. It follows that we can compute $\norm{x}^2 = \product{x}{x}$, and so $\delta^w_{\ell^2} \equiv \delta_{\ell^2}^{=}$, which entails that in fact $\delta^w_{\ell^2}$ is admissible with respect to the (strictly stronger) norm topology.

The flaw in the argument seems to be the claim that for every \emph{compatible} representation $\delta$ of a separable Banach space $X$, the dual representation $\delta' = [\delta \to \delta_{\Field}]\big|^{X'}$ is admissible with respect to the weak* topology on $X'$, the reasoning being that for represented topological spaces $A$ and $B$, the canonical function space representation $[\delta_A\to\delta_B]$ is admissible with respect to the sequentially-compact-open topology on $\RCo{A}{B}$, and that weak*-convergence on $X'$ coincides with compact-open-convergence on $\Co(X,\Field)$ (cf.~Theorem \ref{thm: admissibility of delta E'_w} and Proposition 1 in \cite{BrattkaSchroeder}). However, if $\delta$ is admissible with respect to the weak topology on $X$, then $\delta' = [\delta\to\rho]$ is admissible with respect to the weakly-compact-open topology and not necessarily with respect to the (norm-)compact-open topology. Thus, if $X$ is a reflexive computable Banach space, and we start with the standard representation $\delta_X$ of $X$, which is compatible and admissible with respect to the strong topology, then $\delta'_X = [\delta_{X}\to\rho]$ is compatible and admissible with respect to the compact-open topology on $X'$, which is just the weak* topology. Applying the construction again, we see that $(\delta'_X)' = [\delta_{X}'\to\rho]$ is compatible and admissible with respect to the weak*-compact-open topology, which in general is strictly stronger than the (norm-)compact-open-topology.
\qed
\end{Remark}

Let us now turn to some special properties of convex sets. \emph{Mazur's lemma} asserts that a convex set is weakly sequentially closed if and only if it is strongly closed.

\begin{theorem}[Mazur's lemma]\label{thm: Mazur}
Let $E$ be a Banach space and $K \subseteq E$ be convex. If $(x_n)_n$ is a sequence in $K$ which converges weakly to $x \in E$, then there exists a sequence of finite convex combinations of the $x_n$'s, converging strongly to $x$.
\end{theorem}

It follows that strongly overt convex sets are weakly overt.

\begin{proposition}\label{prop: strongly overt implies weakly overt}
Let $E$ be a reflexive computable Banach space with computable dual $E'$, such that the mapping $\product{\cdot}{\cdot}$ is $(\delta_{E}\times\delta_{E'},\rho)$-computable. Then the identity
\[\id\colon \V^{\conv}(E)\setminus\{\emptyset\} \to \V^{\conv}(E_w)\setminus\{\emptyset\},\]
where $\V^{\conv}(E)$ denotes the hyperspace of convex overt closed subsets of $E$, is well-defined and computable.
\end{proposition}
\begin{proof}
The mapping $\id$ is well-defined by Theorem \ref{thm: Mazur}. Since $E$ is a computable metric space and $E_w$ is separable, we may use the characterisation of overtness given in Proposition \ref{prop: overtness and recursive enumerability, metric case} and the sufficient condition given in Proposition \ref{prop: overtness and recursive enumerability}. If $K \in \V^{\conv}(E)$, then by Proposition \ref{prop: overtness and recursive enumerability, metric case}, we can compute a $\delta_{E}^{\omega}$-name of a norm-dense sequence $(x_n)_n$ in $K$. Since the weak topology is coarser than the norm topology, the weak sequential closure of $(x_n)_n$ contains $K$, and by Theorem \ref{thm: Mazur}, any weak limit of $(x_n)_n$ is already contained in $K$, so that $K$ is the closure of $(x_n)_n$ with respect to the sequentialisation of the weak topology. Since $\product{\cdot}{\cdot}$ is computable, we have $\delta_{E} \leq \delta_{E_w}$, so that we can compute a $\delta_{E_w}^{\omega}$-name of $(x_n)_n$. Thus we can compute $K$ as an element of $\V^{\conv}(E_w)$ using Proposition \ref{prop: overtness and recursive enumerability}.
\end{proof}

Propositions \ref{prop: strongly overt implies weakly overt} and \ref{prop: overtness and recursive enumerability} imply that a convex subset of a reflexive computable Banach space $E$ is weakly overt if and only if it has a computable norm-dense sequence.

Finally, we prove a useful uniform characterisation of computably weakly compact convex sets in a reflexive Banach space $E$ with computable dual $E'$, which will be an important ingredient for the proof of our main result. Note that by Mazur's lemma, a convex subset of $E$ is weakly compact if and only if it is closed and bounded (cf.~also \cite[Proposition 2.8.1]{Megginson}).
\begin{Definition}
Let $E$ be a reflexive Banach space with computable dual $E'$.
\begin{enumerate}[label=(\roman*)]\item A \emph{rational half space} is a nonempty set of the form
\[h = \{x \in E \;|\; \product{x}{x_h'} + a_h \leq 0 \} \]
where $x_h'$ is a rational point in $E'$ and $a_h \in \Q$. A $\nu_{\HB}$-name of a rational half space is a $\nu_{E'}\times\nu_{\Q}$-name of $(x_h',a_h) \in E'\times\Q$.
\item Let $K \subseteq E$ be closed, convex and bounded. A $\kappa_{\HB}$-name of $K$ is a $\nu_{\HB}^\omega\times\kappa^{E_w}$-name of all rational half spaces containing $K$ in their interior and a weakly compact set $L \in \K(E_w)$ containing $K$.
\end{enumerate}
\end{Definition}

\noindent Note that by Theorem \ref{thm: unit ball in weak*-topology as a computable metric space} (ii), a $\kappa_{\HB}$-name of a closed, convex and bounded set can be computed from a list of all rational half spaces containing $K$ in their interior and a rational bound on $\sup\{\norm{x} \;\big|\; x \in K\}$. It may not be immediately obvious that $\kappa_{\HB}$ is a well-defined representation. This follows however from the following easy consequence of the Hahn-Banach separation theorem (cf.~e.g.~\cite[Theorem III.2.5]{Werner}).
\begin{lemma}
Let $E$ be a reflexive Banach space with computable dual $E'$. Let $K \subseteq E$ be closed, bounded and convex, let $x \notin K$. Then there exists a rational half space $h$ such that $K \subseteq \intr{h}$ and $x \in \comp{h}$.\qed
\end{lemma}
Obviously, the boundedness condition on $K$ cannot be dropped, as the example of a straight line with irrational slope in $\R^2$ shows. 
\begin{theorem}\label{Hahn-Banach equivalent to co-r.e.}
Let $E$ be a reflexive Banach space with computable dual $E'$. Then we have
$\kappa_{\HB} \equiv \left(\kappa^{E_w}\right)\bigr|^{\K^{\conv}(E_w)}$, where $\K^{\conv}(E_w)$ denotes the space of convex weakly compact subsets of $E_w$.
\end{theorem}
\begin{proof}
$\left[\kappa_{\HB} \leq \left(\kappa^{E_w}\right)\bigr|^{\K^{\conv}(E_w)}\right]$: Suppose we are given a $\kappa_{\HB}$-name of $K \in \K^{\conv}(E_w)$. Since the name provides us with a weakly compact set $L \in \K(E_w)$ containing $K$, and since $\id\colon \A(L_w)\to \K(L_w)$ is computable by Proposition \ref{prop: closed and compact sets, (i) closed set is compact, (ii) compact set in Hausdorff space is closed} (i), it suffices to show that we can compute a $\psi^{E_w}$-name of $K$. Given $x \in E_w$ and a $\nu_{\HB}$-name of a functional $f = \product{\cdot}{x_h'} + a_h$ we can compute $f(x) \in \R$. Now, if the sequence of half spaces containing $K$ given by the $\kappa_{\HB}$-name is defined by the sequence of affine linear functionals $(f_n)_n$, we can compute the characteristic function of $K^C$ into Sierpi\'nski-space as follows: given $x \in E_w$, if there exists an $n \in \N$ such that $f_n(x) > 0$ output one, otherwise output zero. This shows that we can compute a $\psi^{E_w}$-name of $K$, which proves the claim. \\
$\left[\left(\kappa^{E_w}\right)\bigr|^{\K^{\conv}(E_w)} \leq \kappa_{\HB}\right]$:
Suppose we are given a $\kappa^{E_w}$-name of $K$. We need to compute a weakly compact set $L \in \K(E_w)$ with $L \supseteq K$ and a list of all rational half spaces containing $K$ in their interior. Since $K$ contains itself and is given as a $\kappa^{E_w}$-name, we may put $L = K$, so that it suffices to show that we can enumerate all rational half spaces containing $K$ in their interior. We show that given a $\nu_{\HB}$-name of an affine linear functional (technically, of the half space defining the functional) $f\colon E \to \R$ of the form $f(x) = (x,x') + a$, we can verify if $f(x) < 0$ for all $x \in K$. We can computably translate the $\nu_{\HB}$-name of $f$ into a $\delta_{E'}$-name of $x'$. Then by definition of $E_w$, the mapping $f\colon E_w\to \R, x \mapsto (x,x') + a$ is computable. It follows that $U_f = \{x \in E_w \;\big|\; f(x) < 0\}$ is semi-decidable relative to $f$. By definition of $\kappa$, the relation $K \subseteq U_f$ is semi-decidable relative to $f$ as well, which proves the claim.
\end{proof}

The proof of Theorem \ref{Hahn-Banach equivalent to co-r.e.} shows that the definition of $\kappa_{\HB}$ can be slightly relaxed.

\begin{lemma}\label{lem: listing enough half-spaces}
Let $E$ be a reflexive Banach space with computable dual $E'$. Define a new representation $\tilde{\kappa}_{\HB}$ of $\K^{\conv}(E_w)$ as follows: a $\tilde{\kappa}_{\HB}$-name of $K \in \K^{\conv}(E_w)$ is a $\nu_{\HB}^\omega\times\kappa^{E_w}$-name of a sequence $(h_n)_n$ of rational half spaces such that $K = \bigcap_{n \in \N} \intr{h}_n$ and a weakly compact set $L \in \K(E_w)$ containing $K$. Then $\tilde{\kappa}_{\HB} \equiv \kappa_{\HB}$.
\end{lemma}
\begin{proof}
Clearly, $\kappa_{\HB} \leq \tilde{\kappa}_{\HB}$. For the converse direction, note that the proof of the reduction $\kappa_{\HB} \leq \left(\kappa^{E_w}\right)\bigr|^{\K^{\conv}(E_w)}$ in Theorem \ref{Hahn-Banach equivalent to co-r.e.} actually establishes the stronger reduction ${\tilde{\kappa}_{\HB} \leq \left(\kappa^{E_w}\right)\bigr|^{\K^{\conv}(E_w)}}$, so that we obtain the reduction chain
\[\tilde{\kappa}_{\HB} \leq \left(\kappa^{E_w}\right)\bigr|^{\K^{\conv}(E_w)} \leq \kappa_{\HB} \]
and thus $\tilde{\kappa}_{\HB} \equiv \kappa_{\HB}$. 
\end{proof}

On compact subsets of a Banach space $E$, the weak topology and the norm topology coincide. This is effectively witnessed by our representation. 

\begin{proposition}\label{prop: weak and strong topology equivalent on compact sets}
Let $E$ be a reflexive computable Banach space with computable dual, such that the mapping $\product{\cdot}{\cdot}$ is $(\delta_{E}\times\delta_{E'},\rho)$-computable. Let $K\subseteq E$ be a computably compact and computably overt subset of $E$. Then we have $\delta_{E_w}\big|^K \equiv \delta_{E}\big|^K$.
\end{proposition}
\begin{proof}
Since $\product{\cdot}{\cdot}$ is computable, we have $\delta_{E} \leq \delta_{E_w}$, so we only have to show the converse reduction $\delta_{E_w}|^K \leq \delta_{E}|^K$. Define the represented spaces 
\[K = \left(K,\delta_{E}\big|^K\right)\; \text{ and }\; K_w = \left(K,\delta_{E_w}\big|^K\right).\] 
Firstly, observe that $E_w$ is effectively Hausdorff, i.e. the mapping \[E_w \to \A(E_w), \;x \mapsto \{x\}\] is computable: we can verify if two given elements in $E_w$ are different by comparing their values on the rational points of $E'$. It follows that the identity $\id\colon \O(K) \to \O(K_w)$ is computable via the following chain of maps:

\[
\begin{CD}
\O(K) @>>> \A(K) @>(1)>> \K(K) @>(2)>> \K(K_w) @>(3)>> \A(K_w) @>>> \O(K_w)\\
U @>>> K\setminus U @>>> K\setminus U @>>> K_w\setminus U @>>> K_w\setminus U @>>> U.
\end{CD}
\]\medskip

\noindent The computability of (1) follows from the computable
compactness of $K$ together with Proposition \ref{prop: closed and
  compact sets, (i) closed set is compact, (ii) compact set in
  Hausdorff space is closed} (i). The computability of (2) can be derived from the computability of ${\id\colon\O(E_w)\to\O(E)}$, which in turn follows from the computability of $\id\colon E\to E_w$, and the computability of (3) follows from the fact that $E_w$ is effectively Hausdorff, together with Proposition \ref{prop: closed and compact sets, (i) closed set is compact, (ii) compact set in Hausdorff space is closed} (ii).

We then obtain the mapping $\id\colon K_w \to K$, i.e.~the reduction $\delta_{E_w}|^K \leq \delta_{E}|^K$, via the following chain of maps:
\[
\begin{CD}
K_w @>(4)>> \A(K_w) @>(5)>> \K(K_w) @>(6)>> \K(K) @>(7)>> K\\
x @>>> \{x\} @>>> \{x\} @>>> \{x\} @>>> x.
\end{CD}
\]
Mapping (4) is computable since $E_w$ is effectively Hausdorff. To establish the computability of (5), observe that the computability of (2) and the computable compactness of $K$ imply that $K \in \K(K_w)$ and apply Proposition \ref{prop: closed and compact sets, (i) closed set is compact, (ii) compact set in Hausdorff space is closed} (i). The computability of (6) follows from the computability of $\id\colon \O(K)\to\O(K_w)$, which we have established above. For the computability of (7), observe that we can verify if a rational ball of the form $B(a,2^{-n})$ contains $\{x\}$, which yields a Cauchy sequence effectively converging to $x$ by exhaustive search over all rational balls.
\end{proof}

Finally, we observe that computably overt, co-semi-decidable subsets of $E'_w$ are (uniformly) located. The following proposition guarantees that this actually makes sense.
\begin{proposition}
Let $E$ be a Banach space. Let $A \subseteq E$ be weakly sequentially closed. Then $A$ is closed with respect to the norm topology.
\end{proposition}
\begin{proof}
Since $E$ is a metric space, $A$ is closed if and only if it is sequentially closed. Let $(x_n)_n$ be a sequence in $A$ with limit $x \in E$. Then $x$ is a weak limit of $(x_n)_n$, so $x \in A$, since $A$ is weakly sequentially closed. It follows that $A$ is sequentially closed, and thus closed.
\end{proof}

\begin{proposition}\label{prop: computably weakly closed uniformly implies lower semi-located}
Let $E$ be a reflexive computable Banach space with computable dual, such that the mapping $\product{\cdot}{\cdot}$
is $(\delta_{E}\times\delta_{E'},\rho)$-computable. Then the canonical embedding $i\colon \A(E_w)\setminus\{\emptyset\} \to \A_{\dist_{<}}(E), A \mapsto A$ is computable.
\end{proposition}
\begin{proof}
Given a sequentially weakly closed set $A \in \A(E_w)$, it suffices to show that we can uniformly computably enumerate all closed balls with rational centres and radii contained in the complement $A^C$ of $A$. The result then follows from \cite[Theorem 3.9 (1)]{BrattkaPresser}. The proof of Theorem \ref{thm: unit ball in weak*-topology as a computable metric space} (ii) allows us to uniformly translate a computable number $r \in \R$ into a name of $\clos{B}(0,r)$ as a weakly compact subset of $E_w$, i.e. the mapping
\[(0,\infty) \to \K(E_w), \; r \mapsto \clos{B}(0,r)\]
is computable. It is easy to see that the mapping
\[E_w\times E_w \to E_w,\; (x,c) \mapsto x + c\] 
is computable. Hence, the mapping
\[E_w \times (0,\infty) \to \K(E_w),\; (c,r) \mapsto \clos{B}(c,r) \]
is computable. It follows that the mapping
\[\A(E_w)\times E_w \times (0,\infty) \to \S,\; (A,c,r) \mapsto \begin{cases}1 &\text{if }\clos{B}(c,r) \subseteq A^C\text{,}\\0 &\text{otherwise}\end{cases} \]
is computable. Since $\product{\cdot}{\cdot}$ is computable, we have $\delta_{E}\leq\delta_{E_w}$, so that in particular the mapping
\[\A(E_w)\times E \times (0,\infty) \to \S,\; (A,c,r) \mapsto \begin{cases}1 &\text{if }\clos{B}(c,r) \subseteq A^C\text{,}\\0 &\text{otherwise}\end{cases} \]
is computable. Using this mapping we can enumerate all rational closed balls contained in the complement of $A$.
\end{proof}

Proposition \ref{prop: computably weakly closed uniformly implies lower semi-located} in particular implies that any nonempty weakly co-semi-decidable subset of $E$ is lower semi-located, and hence every nonempty weakly co-semi-decidable and computably overt subset of $E$ is located (which by Proposition \ref{prop: equivalence lower semi-located and computably closed implies locally compact} is at least not uniformly true for co-semi-decidable subsets of $E$, if $E$ is infinite dimensional).

Let us introduce some further Weihrauch degrees. Let $E$ be a computable Banach space with computable dual, such that the mapping $\product{\cdot}{\cdot}$ is $(\delta_{E}\times\delta_{E'},\rho)$-computable. Let $A\subseteq E$ be nonempty and weakly closed. The \emph{weak closed choice principle} $\C_A^{\w\to\w}$ on $A$ is the closed choice principle $\C_{A_w}$ on the represented space $A_w = (A,\delta_{E_w})$. The \emph{weak-strong closed choice principle} is the multimapping
\[\C_A^{\w\to\n}\colon \A(A_w) \rightrightarrows A, \; S \mapsto S,\]
where the image is represented by $\delta_{E}|^A$. Similarly, we define $\XC_A^{\w\to\w}$, $\UC^{\w\to\w}_A$, $\XC^{\w\to\n}_A$ and $\UC^{\w\to\n}_A$. We may also define a (computationally) weaker version of the Browder-Göhde-Kirk theorem. Let $K$ be nonempty, computably overt, weakly co-semi-decidable, bounded and convex. The \emph{weak Browder-Göhde-Kirk} theorem is the mapping \[\WBGK_{K}\colon \Ne(K) \rightrightarrows K_w,\; f\mapsto \Fix(f),\]
where we are given a nonexpansive mapping like in the case of the Browder-G\"ohde-Kirk theorem, but are only required to compute a fixed point with respect to the weak topology. Note that in $\ell^2$ this amounts to computing a fixed point with respect to an orthonormal basis, but not necessarily computing its $\ell^2$-norm (cf.~also \cite{BrattkaInseparable}).

\section{Characterisation of the Fixed Point Sets of Computable Nonexpansive Mappings in Computable Hilbert Space}\label{Section: 5}

We may now prove our main result. Throughout this section we will work on a computable Hilbert space $H$. Note that in this case $H' \simeq H$ is again a computable Hilbert space, and that the mapping $\product{\cdot}{\cdot}\colon H\times H' \to \R$ is the usual inner product on $H$, which is computable by the polarisation identity. In particular, we can use Definition \ref{Def: admissible representation for weak* topology} to construct the space $H_w$, whose representation is admissible for the weak topology on $H$.
\begin{theorem}\label{main result}
Let $H$ be a computable Hilbert space, let $K \subseteq H$ be weakly co-semi-decidable, computably overt, bounded, and convex. Then
\begin{enumerate}[label=(\roman*)]
\item The mapping
\[\Fix\colon \Ne(K)\to \K^{\conv}(K_w)\setminus\{\emptyset\}, f \mapsto \Fix(f)\]
is computable.
\item And so is its multivalued inverse
\[\Fix^{-1}\colon \K^{\conv}(K_w)\setminus\{\emptyset\} \rightrightarrows \Ne(K).\eqno{\qEd}\]
\end{enumerate}
\end{theorem}
Let us sketch the proof of the second claim. Given a nonempty, weakly closed, bounded and convex subset $A$ of $K$, by Theorem \ref{Hahn-Banach equivalent to co-r.e.} we can enumerate a sequence of half spaces whose intersection is equal to $A$. Now, the projections onto these half spaces are nonexpansive, thanks to Theorem \ref{thm: projection theorem} (ii), and computable:
\begin{lemma}\label{lem: computability of projection onto rational half space}
Let $H$ be a computable Hilbert space. There exists a computable function which takes as input a rational half space $h \subseteq H$, encoded as a $\nu_{\HB}$-name, and returns as output the metric projection onto $h$ as an element of $\Co(H,H)$.
\end{lemma}
\begin{proof}
Let $h = \{x \in H \;|\; \product{x}{x_h} + a_h \leq 0\}$, where $x_h$ is a rational point in $H$ and $a_h \in \Q$. It follows from Lemma \ref{lem: effective independence lemma} that we can without loss of generality assume that the set $\{n\in\N \;|\; \nu_H(n) = 0\}$
is \emph{decidable} (cf.~also \cite[Lemma 3]{BrattkaHahnBanach}). Thus, we can decide if $x_h = 0$, and if this is the case we necessarily have $a_h = 0$ (since $h$ is nonempty), and the projection onto $h$ is the identity on $H$. If $x_h \neq 0$, put $\tilde{x}_h = \tfrac{x_h}{\norm{x_h}}$, $\tilde{a}_h = \tfrac{a_h}{\norm{x_h}}$, and $p = x - \alpha \tilde{x}_h$, where $\alpha = \max\{0, \product{x}{\tilde{x}_h} + \tilde{a}_h\}$. One easily verifies that $p \in h$ and that $p$ satisfies the variational inequality (Theorem \ref{thm: projection theorem} (ii)). It follows that $P_h(x) = p$. This proves the claim.
\end{proof}
We can hence compute a sequence of nonexpansive mappings such that $A$ is the intersection of the fixed point sets of these mappings. The following theorem due to Bruck allows us to construct a single nonexpansive mapping whose fixed point set is the intersection of the fixed point sets of our sequence of mappings.
\begin{theorem}[\cite{Bruck}]\label{thm: intersection of fixed points}
Let $E$ be a strictly convex real normed space, let $K \subseteq E$ be nonempty, closed, bounded, and convex.
Let $(\lambda_n)_n$ be any sequence in $(0,1)$ satisfying $\sum_n \lambda_n = 1$. Let $(f_n)_n$ be a family of nonexpansive mappings on $K$ with $\bigcap_{n \in \N}\Fix(f_n) \neq \emptyset$. Then the mapping
\[ f = \sum_n \lambda_nf_n \]
is well-defined, nonexpansive and satisfies
\[ \Fix(f) = \bigcap_{n \in \N} \Fix(f_n).\eqno{\qEd}\]
\end{theorem}\medskip

\noindent In the final step, we project back onto $K$ in order to construct a self-map of $K$.
\begin{lemma}\label{projecting back}
Let $H$ be a real Hilbert space, $K \subseteq H$ be closed and convex and $f\colon K \to H$ be nonexpansive and suppose that $\Fix(f) \neq \emptyset$. Let $P_K$ denote the metric projection onto $K$. Then $P_K\circ f$ is nonexpansive as well with $\Fix(P_K\circ f) = \Fix(f)$.
\end{lemma}
\begin{proof}
It is clear that $P_K \circ f$ is nonexpansive and that $\Fix(f) \subseteq \Fix(P_K \circ f)$.
Suppose there exists $x \in K$ with $f(x) \neq x$ and $P_K(f(x)) = x$. Let $y$ be some fixed point of $f$. Then
\[\norm{f(x) - f(y)}^2 = \norm{y - x}^2 + \norm{x - f(x)}^2 - 2\product{f(x) - x}{y - x}.\]
By assumption, $x = P_K(f(x))$, so by the variational inequality (Theorem \ref{thm: projection theorem})
\[\product{f(x) - x}{y - x} \leq 0 \; \text{for all }y\in K.\]
We also assumed that $f(x) \neq x$, i.e. $\norm{f(x) - x}^2 > 0$, hence
\[\norm{f(x) - f(y)}^2 > \norm{x - y}^2.\]
Contradicting the assumption that $f$ is nonexpansive.
\end{proof}
\begin{proof}[Proof of Theorem \ref{main result} $(ii)$]
We prove that given a $\kappa_{\operatorname{HB}}$-name $\phi$ of a nonempty, closed, convex subset $A \subseteq K$ we can compute the name of a nonexpansive function $f\colon K \to K$ with $\Fix(f) = A$. The name $\phi$ encodes a sequence of rational half spaces $(h_k)_k$ containing $A$ in their interior. Using Lemma \ref{lem: computability of projection onto rational half space}, given $\phi$ we can compute a $[\delta_H\to\delta_H]^\omega$-name of some sequence $(P_k)_k$ of projections, where $P_k$ is the projection onto the rational half space $h_k$. By Theorem \ref{thm: intersection of fixed points}, the mapping $g = \sum_{k \in \N} 2^{-k - 1}P_k$ will satisfy ${\Fix(g) = \bigcap_{k \in \N} \Fix(P_k) = A}$. By Lemma \ref{projecting back}, the mapping $P_K \circ g\colon K \to K$ will have the same set of fixed points. Note that $P_K$ is computable by Corollary \ref{cor: computability of projection on E}, since $K$ is located by Proposition \ref{prop: computably weakly closed uniformly implies lower semi-located}.
\end{proof}
In order to prove item (i) of Theorem \ref{main result}, we need to inspect Lemma \ref{projecting back} a little closer. We first need another simple lemma.
\begin{lemma}\label{lem: convex intersection lemma}
Let $E$ be real normed space, let $C \subseteq E$ be closed and convex, let $S \subseteq C$ be a dense subset of $C$ and let $h$ be a half space in $E$. If $C\cap \intr{h}$ is nonempty, then $S\cap \intr{h}$ is dense in $C \cap h$.
\end{lemma}
\begin{proof}
Let $x \in C\cap h$, and let $c \in C \cap \intr{h}$. Then the line segment joining $c$ and $x$ is contained in $C \cap h$ and contains an element $b \in B(x,\varepsilon/2) \cap \intr{h}$. Now we choose $a \in S$ in a sufficiently small ball around $b$, so that $d(x,a) < \varepsilon$. 
\end{proof}
Lemma \ref{lem: convex intersection lemma} guarantees that intersections of weakly closed and overt sets and rational half spaces are (uniformly) overt. This is a special property, as in general the intersection operator on closed sets is $\left((\psi\sqcap\upsilon)\times(\psi\sqcap\upsilon),\upsilon\right)$-discontinuous (cf.~\cite[Theorem 5.1.13]{Weih}).
\begin{corollary}\label{computability of intersection of computable convex sets and rational half spaces}
Let $E$ be a computable Banach space with computable dual $E'$, such that the mapping $\product{\cdot}{\cdot}$ is $(\delta_{E}\times\delta_{E'},\rho)$-computable. Then intersection of weakly closed convex sets $C$ and closed rational half spaces $h$ with $C\cap \intr{h} \neq \emptyset$ is $\left((\psi^{E_w}\sqcap\upsilon^{E_w})\times\nu_{\HB},(\psi^{E_w}\sqcap\upsilon^{E_w})\right)$-computable, and hence ${\left((\psi^{E_w}\sqcap\upsilon^{E_w})\times\nu_{\HB},\psi_{\dist}\right)}$-computable.
\end{corollary}
\begin{proof}
Since $\nu_{\HB} \leq \psi$, we can always uniformly compute a $\psi$-name of the intersection. In order to compute an $\upsilon$-name, enumerate all elements given by the $\upsilon$-name of $C$ which are also contained in $\intr{h}$. The above lemma guarantees that this yields an $\upsilon$-name of $C \cap h$. The second claim follows from $(\psi^{E_w}\sqcap\upsilon^{E_w}) \equiv \psi_{\dist}$, which in turn follows from Proposition \ref{prop: computably weakly closed uniformly implies lower semi-located}.
\end{proof}

\begin{lemma}\label{quantitative projection lemma}
Let $H$ be a real Hilbert space, $K \subseteq H$ be closed, bounded and convex, let $f\colon K \to K$ be nonexpansive, $h$ be a half space such that $\intr{h}\cap K \neq \emptyset$ and let $S\subseteq K$ be dense in $K$. Let $A = h\cap K$. Then ${\Fix(f) \cap h = \emptyset}$ if and only if
\begin{equation}\label{eq: quantitative projection lemma equation}
\exists x \in \intr{h} \cap S.\exists n \in \N.\left(\norm{f(x) - x} > 2^{-n} \land \norm{P_A(f(x)) - x} < \frac{2^{-2n-3}}{B}\right), 
\end{equation}
where $B \geq \sup\{\norm{x} + 1 \;\big|\; x \in K\}$.
\end{lemma}
\proof
Let us first prove the forward direction. By the Browder-Göhde-Kirk Theorem, ${P_A\circ f|_A\colon A\to A}$ has a fixed point $\tilde{x} \in h \cap K$. Since, by assumption, $\tilde{x}$ is not a fixed point of $f$, there exists an $m \in \N$ with $\norm{f(\tilde{x}) - \tilde{x}} > 2^{-m}$. By Lemma \ref{lem: convex intersection lemma}, $S \cap \intr{h}$ is dense in $h\cap K$, so that we may choose $x \in S\cap\intr{h}$ with $\norm{x - \tilde{x}} < \frac{2^{-2m-6}}{B}$. Then
\begin{align*}
\norm{P_A(f(x)) - x} &\leq \norm{P_A(f(\tilde{x})) - \tilde{x}} + \norm{P_A(f(x)) - P_A(f(\tilde{x}))} + \norm{x - \tilde{x}}\\ &\leq 2 \norm{x - \tilde{x}} < \frac{2^{-2m-5}}{B} = \frac{2^{-2(m + 1) - 3}}{B}.
\end{align*}
And (using $B \geq 1$)
\[\norm{f(x) - x} \geq \norm{f(\tilde{x}) - \tilde{x}} - 2\norm{x - \tilde{x}} > 2^{-m} -  \frac{2^{-2m-5}}{B} \geq 2^{-m} -  2^{-2m-5} > 2^{-m-1}.\]
For the converse direction, we proceed by contrapositive. We suppose that there exists $y \in \Fix(f) \cap h$ (and hence $y \in \Fix(P_A\circ f|_A)$) and show 
\[\forall x \in \intr{h} \cap S.\forall n \in \N. \left( \norm{f(x) - x} > 2^{-n} \rightarrow \norm{P_A(f(x)) - x} \geq \frac{2^{-2n-3}}{B} \right). \]
Let $x \in \intr{h} \cap S$ with $\norm{f(x) - x} > 2^{-n}$. Since $f$ is nonexpansive, we have
\begin{align*}\norm{y - x}^2 &\geq \norm{f(y) - f(x)}^2 \\&= \norm{f(y) - P_A(f(x))}^2 + \norm{P_A(f(x)) - f(x)}^2  + 2(y - P_A(f(x)),P_A(f(x)) - f(x)).
\end{align*}
Now, by the variational inequality, $2(y - P_A(f(x)),P_A(f(x)) - f(x)) \geq 0$, so that
\begin{align*}
\norm{y - x}^2 &\geq \norm{f(y) - P_A(f(x))}^2 + \norm{P_A(f(x)) - f(x)}^2\\
 &=  \norm{y - x}^2 + \norm{x - P_A(f(x))}^2 + 2(y - x,x - P_A(f(x)))\\ &+ \norm{P_A(f(x)) - x}^2 + \norm{x - f(x)}^2 + 2(P_A(f(x)) - x, x - f(x)),
\end{align*}
which entails that
\begin{align*}
0 &\geq \norm{f(x) - x}^2 + 2(P_A(f(x)) - x, x - f(x) - y + x)\\
  &\geq \norm{f(x) - x}^2 - 2\norm{P_A(f(x)) - x}\cdot\norm{x - f(x) - y + x}\\
  &\geq \norm{f(x) - x}^2 - 8\norm{P_A(f(x)) - x}B,
\end{align*}
and hence
\[\norm{P_A(f(x)) - x} \geq \frac{2^{-2n-3}}{B}.\eqno{\qEd}\]\medskip

\noindent Note that it follows from Corollary \ref{computability of intersection of computable convex sets and rational half spaces} that the projection onto $A$ in Lemma \ref{quantitative projection lemma} is computable.

\begin{proof}[Proof of Theorem \ref{main result} $(i)$]
Given a nonexpansive mapping $f\colon K \to K$, we want to compute a $\kappa_{\HB}$-name of $\Fix(f)$. We need to compute a weakly compact set $L \in \K(E_w)$ with $L \supseteq \Fix(f)$, and a list of all rational half spaces containing $f$. Since $K$ contains $\Fix(f)$ and is computably weakly compact by Theorem \ref{thm: unit ball in weak*-topology as a computable metric space} (ii) we may put $L = K$, so it suffices to list all rational half spaces containing $\Fix(f)$. In fact, by Lemma \ref{lem: listing enough half-spaces} it suffices to compute a list of rational half spaces $(h_n)_n$ satisfying $\bigcap_{n \in \N} \intr{h}_n = \Fix(f)$. In order to do so, we enumerate two different lists $L_1$ and $L_2$ of half spaces and interleave them. The first list $L_1$ consists of all rational half spaces containing $K$ in their interior. This list is computable since $K$ is computably weakly compact, and hence $\kappa_{\HB}$-computable by Theorem \ref{Hahn-Banach equivalent to co-r.e.}. In order to compute the second list $L_2$, we first enumerate all rational half spaces $h$ such that $\intr{h} \cap K \neq \emptyset$ and $h^C \cap K \neq \emptyset$. This is possible because $K$ is computably overt. Out of these half spaces we only enumerate those which satisfy $\Fix(f) \cap (\intr{h})^C = \emptyset$. In order to verify this property we apply Lemma \ref{quantitative projection lemma} to the half space $(\intr{h})^C$. Note that we can compute the projection onto $K\cap (\intr{h})^C$ by Corollary \ref{computability of intersection of computable convex sets and rational half spaces}, so that the property \eqref{eq: quantitative projection lemma equation} in Lemma \ref{quantitative projection lemma} becomes semi-decidable. Now, $\Fix(f) \cap (\intr{h})^C = \emptyset$ is equivalent to $\intr{h} \supseteq \Fix(f)$, and it is easy to see that the list $(h_n)_n$ we obtain by interleaving $L_1$ and $L_2$ satisfies $\bigcap_{n \in \N} \intr{h}_n = \Fix(f)$.
\end{proof}

Theorem \ref{main result} now allows us to determine the Weihrauch degree of the weak and strong Browder-Göhde-Kirk theorem.

\begin{theorem}\label{thm: BGK equivalent convex choice}
Let $H$ be a computable Hilbert space and $K\subseteq H$ be nonempty, bounded, convex, computably weakly closed, and computably overt. Then \[\BGK_K \equiv_W \XC_K^{\w\to\n},\] and \[\WBGK \equiv_W \XC_K^{\w\to\w}.\] If $K$ is computably compact, then \[\BGK_K \equiv_W \WBGK_K \equiv_W \XC_K.\]
\end{theorem}
\begin{proof}
The equivalences $\BGK_K \equiv_W \XC_K^{\w\to\n}$ and $\WBGK \equiv_W \XC_K^{\w\to\w}$ follow from Theorem \ref{main result}, together with the fact that by Theorem \ref{thm: unit ball in weak*-topology as a computable metric space} $K$ is computably weak* compact, so that $\id:\A^{\conv}(K) \to \K^{\conv}(K)$ is computable. If $K$ is computably compact, then by Proposition \ref{prop: weak and strong topology equivalent on compact sets} we have $\delta_{E}\big|^K \equiv \delta_{E_w}\big|^K$, which yields $\BGK_K \equiv_W \WBGK_K \equiv_W \XC_K$ (note that if two representations of the same space are equivalent, then the induced canonical representations of closed sets are - by construction - equivalent as well). 
\end{proof}

Theorem \ref{thm: BGK equivalent convex choice} also shows that on a non-compact domain, negative information on the weak closedness of a set is much stronger than negative information on its norm-closedness. We have for instance $\BGK_{B_{\ell^2}} \leq_W \Proj_{B_{\ell^2}} \equiv_W \lim$, and so $\XC_{B_{\ell^2}}^{\w\to\n} \leq_W \lim$, while already $\UC_{B_{\ell^2}}^{\n\to\n}$ is equivalent to the extremely non-effective principle $\C_{\N^\N}$. In finite dimension, the degree of $\BGK_K$ is always strictly below $\WKL$ because of Corollary \ref{cor: fixed points in finite dimension are computable}. On the unit ball in $\ell^2$ this is no longer the case.

\begin{theorem}\label{thm: computable mapping with norm-uncomputable image}
There exists a computable mapping
\[T\colon \subseteq [0,1]^{\N} \to \K^{\conv}(B_{\ell^2_w})\]
with $\dom T = \{x \in [0,1]^\N\;|\;x(n)\leq x(n + 1)\}$ such that for all $x \in \dom T$ we have $T(x) = \{a\}$ with ${\norm{a}_2 = \lim x(n)}$.
\end{theorem}
\begin{proof}
Let $x\in \dom T$. Put $a(0) = x(0)$ and $a(n + 1) = \sqrt{x(n + 1)^2 - x(n)^2}$. Then we have $a(n)^2 + \dots + a(0)^2 = x(n)^2$. Now, put $T(x) = \{a\}$. Note that $a$ is $\delta_{\ell^2_w}$-computable relative to $x$, so we can compute $\{a\}$ in $\K^{\conv}(B_{\ell^2_w})$: in order to compute the characteristic function of $\{a\}^C$ into Sierpi\'nski space we simply check for inequality with $a$ component-wise. This allows us to compute $\{a\}$ as a point in $\A^{\conv}(B_{\ell^2_w})$, and thus as a point of $\K^{\conv}(B_{\ell^2_w})$, using that the identity $\id\colon \A^{\conv}(B_{\ell^2_w}) \to \K^{\conv}(B_{\ell^2_w})$ is computable, since $B_{\ell^2_w}$ is computably weakly compact by Theorem \ref{thm: unit ball in weak*-topology as a computable metric space}.
\end{proof}

Choosing from a set in $\K^{\conv}(B_{\ell^2})$ hence allows us computably translate a $\rho_<$-name to a $\rho$-name of a given real number $x \in [0,1]$, already if the set is a singleton. This yields:

\begin{corollary}\label{cor: Weihrauch degree of BGK on unit ball}
$\BGK_{B_{\ell^2}} \equiv_W \UC_{B_{\ell^2}}^{\w\to\n} \equiv_W \lim$. In particular $\BGK_{B_{\ell^2}} \equiv_W \Proj_{B_{\ell^2}}$.
\end{corollary}
\begin{proof}
We have $\UC_{B_{\ell^2}}^{\w\to\n} \leq_W \XC_{B_{\ell^2}}^{\w\to\n} \equiv_W \BGK_{B_{\ell^2}}$, the latter by Theorem \ref{thm: BGK equivalent convex choice}. Also, $\BGK_{B_{\ell^2}} \leq_W \lim$ by Proposition \ref{prop: obvious Weihrauch reductions} and Proposition \ref{prop: initial upper bound on Proj}. It follows from Theorem \ref{thm: computable mapping with norm-uncomputable image} that $\UC_{B_{\ell^2}}^{\w\to\n}$ allows us to determine the limit of any computable monotonically increasing sequence $x \in [0,1]^{\N}$, since
\[\lim_{n\to\infty} x(n) = \norm{\UC_{B_{\ell^2}}^{\w\to\n}(T(x))}_2\]
and $\norm{\cdot}_2$ is $(\delta_{\ell^2},\rho)$-computable. We hence have $\lim \leq_W \UC_{B_{\ell^2}}^{\w\to\n}$, which finishes the proof.
\end{proof}

In particular we have the following non-uniform corollary:
\begin{corollary}\label{cor: nonuniform complexity of BGK on unit ball}
There exists a computable nonexpansive self-map of the closed unit ball in $\ell^2$ with a unique fixed point, which is uncomputable.\qed
\end{corollary}

Compare Corollary \ref{cor: nonuniform complexity of BGK on unit ball} to Theorem \ref{thm: nonexpansive mapping on Hilbert cube without computable fixed points}: on a compact domain, any computable function without computable fixed points necessarily has uncountably many fixed points, since otherwise it has at least one isolated fixed point which is then computable by Theorem \ref{thm: Kreinovich's theorem}. If we drop compactness, even unique solutions may be uncomputable. Note however, that since the unit ball in $\ell^2$ is still computably weakly compact, unique fixed points on $B_{\ell^2}$ are still ``weakly computable'', in the sense that they are computable as elements in the represented space $\ell^2_w$. In particular, their coordinates with respect to an orthonormal basis are still computable.

On a computably compact domain, the Weihrauch degree of the theorem is still at most $\WKL$. We can now show that it is in fact equivalent to $\WKL$ on the Hilbert cube. In order to do so, we will first have to define the \emph{parallelisation} of Weihrauch degrees, which was introduced in \cite{WeihrauchDegrees}.

\begin{Definition}
Let $f\colon\subseteq X \rightrightarrows Y$ be a partial multimapping. The \emph{parallelisation} $\hat{f}$ of $f$ is the multimapping $\hat{f}\colon\subseteq X^\N\rightrightarrows Y^\N$, $\hat{f}(\lambda n.x(n)) = \lambda n.f(x(n))$.
\end{Definition}

It is not hard to see that $f \leq_W g$ implies $\hat{f} \leq_W \hat{g}$ (cf.~also Proposition 4.2 in \cite{WeihrauchDegrees}). The following theorem is essentially due to \cite{WeihrauchDegrees} (cf.~also Theorem 6.2 and the subsequent comment in \cite{EffectiveChoice}).

\begin{theorem}
$\widehat{\IVT} \equiv_W \WKL$.\qed
\end{theorem}

\begin{theorem}\label{thm: BGK on Hilbert cube is WKL}
Let $\mathcal{H} = \{\sum_{i \in \N} \alpha_i e_i\;|\; \alpha_i \in [0,2^{-i}]\}$ be the Hilbert cube in $\ell^2$. Then $\BGK_{\mathcal{H}} \equiv_W \WKL$.
\end{theorem}
\begin{proof}
Clearly, $\mathcal{H}$ is computably compact, so $\BGK_{\mathcal{H}} \equiv_W \XC_{\mathcal{H}} \leq_W \C_{\mathcal{H}} \equiv_W \WKL$. In order to prove the converse direction, we show that $\widehat{\IVT} \leq_W \XC_{\mathcal{H}}$. Since $\widehat{\IVT} \equiv_W \WKL$ and $\XC_{\mathcal{H}} \equiv_W \BGK_{\mathcal{H}}$, it follows that $\WKL \leq_W \BGK_{\mathcal{H}}$. By Proposition \ref{prop: BGK in [0,1]} we have $\IVT \equiv_W \XC_{[0,1]}$ and so $\widehat{\IVT} \equiv_W \widehat{\XC_{[0,1]}}$. Let $([a_n,b_n])_n$ be a sequence of closed intervals in $\left(\A^{\conv}([0,1])\right)^{\N}$. Consider the set $A = \{\sum_{i \in \N} \alpha_i e_i\;|\; \alpha_i \in [a_i2^{-i},b_i2^{-i}]\} \subseteq {\mathcal{H}}$. Then $A$ is computable as a point in $\K^{\conv}({\mathcal{H}})$ relative to $([a_n,b_n])_n$. Clearly, choosing a point in $A$ allows us to choose a point in $([a_n,b_n])_n$, so $\widehat{\XC_{[0,1]}} \leq_W \XC_{\mathcal{H}} \equiv_W \BGK_\mathcal{H}$.
\end{proof}

Theorem \ref{thm: BGK on Hilbert cube is WKL} can (essentially) be viewed as a uniform strengthening of Theorem \ref{thm: nonexpansive mapping on Hilbert cube without computable fixed points}. Notice that the proof of Theorem \ref{thm: nonexpansive mapping on Hilbert cube without computable fixed points} can be utilized to establish the reduction ${\widehat{\LLPO} \leq_W \BGK_{\mathcal{H}}}$, which yields a slightly different proof of Theorem \ref{thm: BGK on Hilbert cube is WKL}, since $\widehat{\LLPO} \equiv_W \WKL$ (again, cf.~\cite{WeihrauchDegrees}). The proof of Theorem \ref{thm: BGK on Hilbert cube is WKL} can also be used to show $\WBGK_{B_{\ell^2}} \equiv_W \WKL$. We now have a fairly good idea of the computational content of the Browder-Göhde-Kirk theorem. It follows from \cite{PaulyLeRoux} that $(\BGK_{[0,1]^n})_{n \in \N}$ is a strictly increasing sequence of Weihrauch degrees, all strictly below $\WKL$. On the compact but infinite dimensional Hilbert cube $\mathcal{H}$ the theorem becomes equivalent to $\WKL$. If we drop compactness and consider the theorem on the unit ball in $\ell^2$, it becomes even more non-effective, and in particular equivalent to computing rates of convergence for fixed point iterations, but is still much more effective than full choice on $B_{\ell^2}$.

In finite dimension, a computable nonexpansive self-map of a computably compact domain always has computable fixed points by Theorem \ref{thm: convex sets have computable points}, and this relies solely on the fact that the fixed point set is convex. This is reminiscent of the fact that unique zeroes of computable functions are always (in this case even uniformly) computable. A typical feature of such results is that they assert the existence of computable objects, but the computational complexity of these objects is unbounded. This is also the case here: using similar techniques as in \cite{Ko}, we can strengthen Theorem \ref{main result} $(ii)$ to assert for every nonempty co-semi-decidable and convex $A \subseteq K$ the existence of a polynomial-time computable nonexpansive $f\colon K\to K$ such that $\Fix(f) = A$, at least in the case where $K$ is computably compact, and so in particular in the finite-dimensional case (if $K$ is not computably compact there is no uniform majorant on the names of the points in $K$, so one would have to work in the framework of \emph{second-order complexity} \cite{KawamuraCook}). This allows us to characterise the computational complexity of fixed points of Lipschitz-continuous polynomial-time computable functions according to their Lipschitz constant.

\begin{theorem}\label{thm: complexity of mappings w.r.t. Lipschitz constant}
Let $[0,1]^2$ be the unit square in Euclidean space $\R^2$. Let ${f\colon [0,1]^2 \to [0,1]^2}$ be polynomial-time computable and Lipschitz-continuous with Lipschitz constant $L$. Then:
\begin{itemize}[label=$-$]
\item If $L < 1$, $f$ has a unique polynomial time computable fixed point, which is uniformly computable relative to the promise that $L < 1$ and uniformly polynomial time computable relative to the promise that $L < 1 - \varepsilon$ for some fixed $\varepsilon > 0$.
\item If $L = 1$, the fixed point set of $f$ can be any nonempty co-semi-decidable convex subset of $[0,1]^2$. The multi-valued operator mapping $f$ to some fixed point is realiser-discontinuous and hence uncomputable, but $f$ still has computable fixed points. However, there is no computable bound on the computational complexity of the fixed points of $f$.
\item If $L > 1$, $f$ may not have any computable fixed points.\qed
\end{itemize}
\end{theorem} 

\noindent The third claim in Theorem \ref{thm: complexity of mappings w.r.t. Lipschitz constant} follows from a strengthening of the results in \cite{ConnectedChoice}, which the authors of that paper have recently obtained, but which seems to be unpublished as of yet. 

\section{Further Results and Possible Generalisations}
The special case of Theorem \ref{main result} where the underlying Hilbert space is two-dimensional seems to generalise to uniformly convex and smooth real Banach spaces of dimension two. Note that the first item of the theorem becomes trivial in finite dimension. A Banach space is called \emph{smooth}, if its dual space is strictly convex and \emph{uniformly smooth} if its dual space is uniformly convex. For instance, all $L^p$-spaces with $1 < p < \infty$ are uniformly convex and uniformly smooth. The two notions of smoothness and uniform smoothness coincide in finite dimension.
\begin{conjecture}\label{characterisation of fixed point sets in two dimensional smooth and convex Banach space}
Let $E$ be a uniformly convex, smooth, computable Banach space of dimension two, and let $K\subseteq E$ be bounded, convex, and located\footnote{Recall that in finite dimension a nonempty closed set is located if and only if it is co-semi-decidable and computably overt (cf.~also \cite{Weih}).}. Then the multi-valued mapping
\[\Fix^{-1}\colon \K^{\conv}(K)\setminus\{\emptyset\}\rightrightarrows \Ne(K)\]
is computable.\qed
\end{conjecture}
The proof of this result would be almost identical to that of Theorem \ref{main result}. The only places where we used that the underlying space is a Hilbert space were Theorem \ref{thm: projection theorem}, which asserts that the projection onto each convex, closed set is nonexpansive, and Lemma \ref{projecting back}. In general the projection onto a closed and convex subset of a Banach space will not be nonexpansive. In fact, this property characterises Hilbert spaces (cf.~\cite{Phelps}). However, we only need the existence of a computable nonexpansive retraction onto each located convex subset. A retraction $Q\colon E \to K$ of $E$ onto a nonempty subset $K\subseteq E$ is called \emph{sunny}, if 
\[Q(\alpha x + (1 - \alpha)Q(x)) = Q(x) \;\text{ for all } x \in E, \alpha \in [0,1].\]
Geometrically, this means that for all $x \notin K$, all points on the ray  defined by $x$ and $Q(x)$ with initial point $Q(x)$ are mapped onto the same point $Q(x)$. It is well known that in a smooth Banach space of dimension two, sunny nonexpansive retractions onto closed convex subsets exist and are unique. Consequently, they are computable.
\begin{theorem}[\cite{Karlovitz}]\label{existence of sunny nonexpansive retracts}
Let $E$ be a smooth real Banach space of dimension two. Then for every nonempty closed convex subset $C$ of $E$, there exists a nonexpansive sunny retraction of $E$ onto $C$.\qed
\end{theorem}

\begin{theorem}[\cite{BruckProjections}]\label{uniqueness of sunny nonexpansive retracts}
Let $E$ be a smooth real Banach space. Let $K \subseteq C$ be two nonempty, closed, and convex subsets of $E$. Then there exists at most one sunny nonexpansive retraction of $C$ onto $K$.\qed
\end{theorem}

\begin{theorem}\label{computability of sunny nonexpansive retracts}
Let $E$ be a smooth computable Banach space of dimension two. Let $C \subseteq E$ be nonempty, convex, bounded, and located. Then the mapping
\[\operatorname{SRet}\colon\A^{\conv}_{\dist}(C)\setminus\{\emptyset\}\mapsto
\Ne(C)\]
 that maps $K$ to the unique sunny nonexpansive retraction of $\,C$ onto
 $K$, 
is computable.
\end{theorem}
\begin{proof}[Proof (sketch)]
The set of nonexpansive self-maps of $C$ is computably compact, since it is equicontinuous and $C$ is compact. We can verify if a given map $f\colon C\to C$ does not leave all points of $K$ fixed, if it maps a point of $C$ to a point outside of $K$, and if ${f(\alpha x + (1 - \alpha)f(x)) \neq f(x)}$ for some $x \in C$, $\alpha \in [0,1]$. It follows that the set of sunny nonexpansive retractions of $C$ onto $K$ is co-semi-decidable relative to (a $\psi_{\dist}$-name of) $K$. Theorems \ref{existence of sunny nonexpansive retracts} and \ref{uniqueness of sunny nonexpansive retracts} assert that it is a singleton. It follows that the operator is uniformly computable.
\end{proof}
The other result that relies on Hilbert space techniques is Lemma \ref{projecting back}, which uses the nonexpansiveness of the projection and the variational inequality. In principle we could replace the projection by the sunny nonexpansive retraction onto the domain, but the question remains whether this will always leave the fixed point set unchanged.
\begin{conjecture}\label{projection lemma in Banach space}
Let $E$ be a smooth and uniformly convex Banach space of dimension two, let $K \subseteq E$ be nonempty, closed, bounded, and convex and let $f\colon K \to E$ be nonexpansive with ${\Fix(f)\neq\emptyset}$. Let $P\colon E \to K$ be the sunny nonexpansive retraction onto $K$. Then we have $\Fix(P\circ f) = \Fix(f)$.
\end{conjecture}
\begin{proof}[Proof of to Conjecture \ref{characterisation of fixed point sets in two dimensional smooth and convex Banach space} up to Conjecture \ref{projection lemma in Banach space}] We could now prove Conjecture \ref{characterisation of fixed point sets in two dimensional smooth and convex Banach space} analogously to Theorem \ref{main result}: we are given a convex, closed subset $A$ of $K$ as a $\kappa_{\HB}$-name and want to construct a nonexpansive mapping $f\colon K \to K$ with $\Fix(f) = A$. Let $(h_n)_n$ be the sequence of half spaces given by the $\kappa_{\HB}$-name. Since $h_n\cap K \neq \emptyset$ for all $n$, we can compute a $\psi_{\dist}^\omega$-name of the sequence $(h_n\cap K)_{n \in \N}$ thanks to Corollary \ref{computability of intersection of computable convex sets and rational half spaces}. Now, Theorem \ref{computability of sunny nonexpansive retracts} allows us to compute a $[\delta_K\to\delta_K]^\omega$-name of the sequence $(f_n)_n$ of sunny nonexpansive retractions of $K$ onto $h_n\cap K$. Applying Theorem \ref{thm: intersection of fixed points}, we obtain a nonexpansive mapping $g\colon K \to E$ with $\Fix(g) = \bigcap_{n\in\N} h_n = A$. Finally, we use the computable nonexpansive sunny retraction onto $K$ and Conjecture \ref{projection lemma in Banach space} to obtain a self-map $f$ of $K$ with $\Fix(f) = A$.
\end{proof}
The only ``missing piece'' in this proof is Conjecture \ref{projection lemma in Banach space}. By replacing this conjecture by a weaker statement that we can prove, we obtain a weaker version of Conjecture \ref{characterisation of fixed point sets in two dimensional smooth and convex Banach space}, which is almost as good.
\begin{lemma}\label{weak projection lemma}
Let $E$ be a uniformly convex Banach space of dimension two, let $K \subseteq E$ be nonempty, closed, bounded, and convex. Suppose that $\intr{K}$ is nonempty and that $\partial K$ does not contain any line segments, and let $f\colon K \to E$ be nonexpansive with $\Fix(f)\neq\emptyset$. Let $P\colon E \to K$ be the sunny nonexpansive retraction onto $K$. Then $P\circ f$ is nonexpansive as well with $\Fix(P\circ f) = \Fix(f)$.
\end{lemma}
\begin{proof}
Clearly, $P\circ f$ is nonexpansive with $\Fix(f) \subseteq \Fix(P \circ f)$. Suppose that there exists $x \in \Fix(P\circ f)$, which is not a fixed point of $f$. Since $P$ is sunny, $x \in \partial K$. Let $y \in \Fix(f)$. Since $\Fix(f)$ is closed, there exists $\varepsilon > 0$ such that $B(x,\varepsilon) \subseteq \Fix(f)^C$. Since $P\circ f$ is nonexpansive, the line segment $L$ joining $y$ and $x$ is contained in $\Fix(P \circ f)$. By hypothesis, the line segment without its endpoints has to lie in $\intr{K}$ (it is easy to see that if a convex set contains three points of a line segment in its boundary, it contains the whole line segment in its boundary). Hence, there exists $z \in \intr{K}\cap L \cap B(x,\varepsilon)$. Contradiction.
\end{proof}
\begin{Remark}
A similar proof shows that we may replace the condition that $\partial K$ contains no line segments, by the condition that $\Fix(f)\cap\intr{K} \neq \emptyset$. In this case we do not even require the retraction to be sunny.
\end{Remark}
\begin{theorem}\label{weak main theorem on Banach spaces}
Let $E$ be a uniformly convex, smooth computable Banach space of dimension two, and let $K\subseteq E$ be nonempty, bounded, convex, and located. Suppose that either $\dim K = 2$ and $\partial K$ contains no line segments or $\dim K = 1$. then the multi-valued mapping
\[\Fix^{-1}\colon \K^{\conv}(K)\setminus\{\emptyset\}\rightrightarrows \Ne(K)\]
is computable.
\end{theorem}
In a uniformly convex space the unit ball contains no line segments, so $B_E$ is an example of an admissible domain $K$. In particular, every co-semi-decidable, convex subset of $B_E$ is the fixed point set of some computable, nonexpansive self-map of $B_E$.
\begin{proof}[Proof of Theorem \ref{weak main theorem on Banach spaces}]
If $\dim K = 1$, we introduce suitable coordinates in which $K$ is contained in the $x$-axis and use the construction of Theorem \ref{thm: main theorem on [0,1]}. If $\dim K = 2$, we use the proof of Conjecture \ref{characterisation of fixed point sets in two dimensional smooth and convex Banach space}. Note that here we may replace Conjecture \ref{projection lemma in Banach space} by Lemma \ref{weak projection lemma}, so the proof is complete.
\end{proof}

The obvious question at this point is whether Conjecture \ref{characterisation of fixed point sets in two dimensional smooth and convex Banach space} might generalise to higher dimensional Banach spaces. While most of the results we used in the proof at least generalise to finite-dimensional smooth and uniformly convex computable Banach spaces the main obstruction appears to be the existence of nonexpansive retractions. Our proof uses the fact that there exist nonexpansive retractions onto every rational half space, but if $E$ is a Banach space of dimension at least three and there exist nonexpansive retractions onto each two-dimensional subspace, then $E$ is a Hilbert space. Similarly, the unit ball of an at least three-dimensional Banach space $E$ is a nonexpansive retract of $E$ if and only if $E$ is a Hilbert space (cf.~\cite{BruckHilbert}). On the other hand, every fixed point set of a nonexpansive mapping $f\colon K\to K$ is a nonexpansive retraction of $K$. In view of these results it seems likely that Theorem \ref{main result} characterises computable Hilbert space of dimension three or higher.

Finally, we extend the stronger upper bound obtained in Proposition \ref{prop: initial upper bound on Proj} for compact sets and Hilbert space to the noncompact case in uniformly convex and uniformly smooth spaces. For this we need a generalisation of Theorem \ref{thm: Halpern's theorem} due to Reich \cite{ReichStrongConvergence}. We will only state a special case.

\begin{theorem}[\cite{ReichStrongConvergence}]\label{thm: Reich's theorem}
Let $E$ be a uniformly smooth, uniformly convex Banach space, let $K \subseteq E$ be nonempty, closed, bounded and convex, let $f\colon K \to K$ be nonexpansive, and let $x \in K$. Put $\alpha_n = 1 - (n + 2)^{-\tfrac{1}{2}}$. Then the sequence $(x_n)_n$ defined by the iteration scheme $x_0 = x$ and
\[x_{n + 1} = (1 - \alpha_n)x_0 + \alpha_n f(x_n) \]
converges to a fixed point of $f$. \qed
\end{theorem}

Note that the iteration defined in Theorem \ref{thm: Reich's theorem} converges to a retraction onto the fixed point set of $f$. In fact, one can show that the sequence $(x_n)_n$ converges to $Q(x_0)$, where $Q$ is the unique sunny nonexpansive retraction of $K$ onto $\Fix(f)$.

\begin{theorem}\label{thm: upper bound on Proj in smooth spaces}
Let $E$ be a uniformly convex, uniformly smooth computable Banach space. Let $K \subseteq E$ be nonempty, bounded, convex, co-semi-decidable, and computably overt. Then
\[\Proj_K \leq_W \lim.\]
\end{theorem}
\begin{proof}
We use similar ideas as in the proof of Proposition \ref{prop: initial upper bound on Proj}. Again we exploit the fact that we can actually compute countably many instances of $\lim$ in parallel. As in the proof of the general upper bound in Proposition \ref{prop: initial upper bound on Proj}, we use countably many instances of $\lim$ to obtain a function $\mu\colon \N \to \N$ satisfying $2^{-\mu(n)} \leq \eta_E(2^{-n})$, where $\eta_E$ is a modulus of uniform convexity for $E$, and another batch of countably many instances to obtain an approximation to the distance function to $\Fix(f)$ from below. Since $K$ is computably overt, it contains a computable dense sequence $(x_n)_n$. Let $x_n^0 = x_n$ and 
$x_n^{k + 1} = (1 - \alpha_k)x_n^0 + \alpha_k f(x_n^k)$ with $\alpha_k$ as in Theorem \ref{thm: Reich's theorem}. Using  another countable batch of instances of $\lim$, we obtain the sequence $(\lim_{k\to\infty} x^k_n)_n$, which is dense in $\Fix(f)$, since the iteration defines a retraction of $K$ onto $\Fix(f)$. Using this sequence we can compute the distance function to $\Fix(f)$ from above, so that we obtain the distance function to $\Fix(f)$ as an element of $\Co(K,\R)$. Together with Corollary \ref{cor: computability of projection on subset} this establishes the reduction. 
\end{proof}

\section*{Acknowledgements.}
The present work was motivated by a question by Ulrich Kohlenbach, whether the Krasnosel\-ski-Mann iteration has nonuniformly computable rates of convergence. He has also provided many valuable insights both concerning fixed point theory and computability theory. This work has greatly benefited from discussions with Vasco Brattka, Arno Pauly, Guido Gherardi, and Martin Ziegler. I would also like to thank the anonymous referees for pointing out many shortcomings in the original version of this paper.

\bibliographystyle{abbrv}
\bibliography{lmcs_final}

\begin{thebibliography}{10}

\bibitem{AroraBarak}
S.~Arora and B.~Barak.
\newblock {\em Computational Complexity: A Modern Approach}.
\newblock Cambridge University Press, 2009.

\bibitem{AvigadEtAl}
J.~Avigad, P.~Gerhardy, and H.~Towsner.
\newblock Local stability of ergodic averages.
\newblock {\em Trans. Amer. Math. Soc.}, 362(1):261--288, 2010.

\bibitem{AvigadRute}
J.~Avigad and J.~Rute.
\newblock Oscillation and the mean ergodic theorem for uniformly convex
  {B}anach spaces.
\newblock {\em Ergodic Theory and Dynamical Systems, Available on CJO 2014,
  doi:10.1017/etds.2013.90}, 2014.

\bibitem{Baigger}
G.~Baigger.
\newblock Die {N}ichtkonstruktivit{\"a}t des {B}rouwerschen {F}ixpunktsatzes.
\newblock {\em Archiv f{\"u}r Mathematische Logik und Grundlagenforschung},
  25:183--188, 1985.

\bibitem{Beeson}
M.~Beeson.
\newblock {\em Foundations of {C}onstructive {M}athematics}.
\newblock Springer-Verlag, New York, 1985.

\bibitem{Bishop}
E.~Bishop and D.~Bridges.
\newblock {\em Constructive Analysis}.
\newblock Springer-Verlag, 1985.

\bibitem{bonsall1962lectures}
F.~F. Bonsall.
\newblock {\em Lectures on some fixed point theorems of functional analysis},
  volume~26 of {\em Tata Institute of Fundamental Research Lectures on
  Mathematics and Physics: Mathematics}.
\newblock Tata Institute of Fundamental Research, 1962.

\bibitem{BrattkaBounds}
V.~Brattka.
\newblock Computing uniform bounds.
\newblock {\em Electronic Notes in Theoretical Computer Science}, 66(1):13 --
  24, 2002.
\newblock CCA 2002, Computability and Complexity in Analysis.

\bibitem{BrattkaInseparable}
V.~Brattka.
\newblock Computability on non-separable {B}anach spaces and {L}andau's
  theorem.
\newblock In L.~Crosilla and P.~Schuster, editors, {\em From Sets and Types to
  Topology and Analysis: Towards Practicable Foundations for Constructive
  Mathematics}, pages 316--333. Oxford University Press, 2005.

\bibitem{BrattkaEffectiveBorel}
V.~Brattka.
\newblock Effective {B}orel measurability and reducibility of functions.
\newblock {\em Mathematical Logic Quarterly}, 51(1):19--44, 2005.

\bibitem{Brattka06computableversions}
V.~Brattka.
\newblock Computable versions of the uniform boundedness theorem.
\newblock In Z.~Chatzidakis, P.~Koepke, and W.~Pohlers, editors, {\em Logic
  Colloquium 2002}, volume~27 of {\em Lecture Notes in Logics}, pages 130--151.
  Association for Symbolic Logic, 2006.

\bibitem{BrattkaHahnBanach}
V.~Brattka.
\newblock Borel complexity and computability of the {H}ahn-{B}anach {T}heorem.
\newblock {\em Archive for Mathematical Logic}, 46(7-8):547--564, 2008.

\bibitem{ClosedChoice}
V.~Brattka, M.~de~Brecht, and A.~Pauly.
\newblock Closed choice and a uniform low basis theorem.
\newblock {\em Annals of Pure and Applied Logic}, 163(8):986 -- 1008, 2012.

\bibitem{BrattkaDillhage}
V.~Brattka and R.~Dillhage.
\newblock Computability of finite dimensional linear subspaces and best
  approximation.
\newblock {\em Annals of Pure and Applied Logic}, 162:182--193, 2010.

\bibitem{BrattkaGherardi}
V.~Brattka and G.~Gherardi.
\newblock Borel complexity of topological operations on computable metric
  spaces.
\newblock {\em Journal of Logic and Computation}, 19(1):45--76, 2009.

\bibitem{EffectiveChoice}
V.~Brattka and G.~Gherardi.
\newblock Effective choice and boundedness principles in computable analysis.
\newblock {\em Bulletin of Symbolic Logic}, 17(1):73--117, 2011.

\bibitem{WeihrauchDegrees}
V.~Brattka and G.~Gherardi.
\newblock Weihrauch degrees, omniscience principles and weak computability.
\newblock {\em J. Symbolic Logic}, 76(1):143--176, 2011.

\bibitem{BolzanoWeierstrass}
V.~Brattka, G.~Gherardi, and A.~Marcone.
\newblock The {B}olzano-{W}eierstrass {T}heorem is the {J}ump of {W}eak {K}{\H
  o}nig's {L}emma.
\newblock {\em Annals of Pure and Applied Logic}, 163(6):623--655, 2012.

\bibitem{ConnectedChoice}
V.~Brattka, S.~Le~Roux, and A.~Pauly.
\newblock Connected {C}hoice and the {B}rouwer {F}ixed {P}oint {T}heorem.
\newblock {\em arXiv:1206.4809v1}, 2012.

\bibitem{BrattkaPresser}
V.~Brattka and G.~Presser.
\newblock Computability on subsets of metric spaces.
\newblock {\em Theoretical Computer Science}, 305(1–3):43 -- 76, 2003.

\bibitem{BrattkaSchroeder}
V.~Brattka and M.~Schr{\"o}der.
\newblock Computing with sequences, weak topologies and the axiom of choice.
\newblock In L.~Ong, editor, {\em Computer Science Logic}, volume 3634 of {\em
  Lecture Notes in Computer Science}, pages 462--476. Springer Berlin
  Heidelberg, 2005.

\bibitem{BrouwerFPIntuitionistic}
L.~Brouwer.
\newblock An intuitionist correction of the fixed point theorem on the sphere.
\newblock {\em Proceedings of the Royal Societey. London. Series A.}, 213:1--2,
  1952.

\bibitem{BrowderFixedPointA}
F.~Browder.
\newblock Fixed point theorems for noncompact mappings in {H}ilbert space.
\newblock {\em Proc. Nat. Acad. Sci. USA}, 43:1272--1276, 1965.

\bibitem{BrowderFixedPointB}
F.~Browder.
\newblock Nonexpansive nonlinear operators in a {B}anach space.
\newblock {\em Proc. Nat. Acad. Sci. USA}, 54:1041--1044, 1965.

\bibitem{BruckHilbert}
R.~Bruck.
\newblock A characterisation of {H}ilbert space.
\newblock {\em Proc. Amer. Math. Soc.}, 43:173--175, 1974.

\bibitem{BruckProjections}
R.~E. Bruck.
\newblock Nonexpansive projections on subsets of {B}anach spaces.
\newblock {\em Pacific Journal Of Mathematics}, 47:341--355, 1973.

\bibitem{Bruck}
R.~E. Bruck.
\newblock Properties of fixed-point sets of nonexpansive mappings in {B}anach
  spaces.
\newblock {\em Trans.Amer.Math.Soc.}, 179:251--262, 1973.

\bibitem{ChidumeMutangadura}
C.~Chidume and S.~Mutangadura.
\newblock An example on the {M}ann iteration method for {L}ipschitz
  pseudocontractions.
\newblock {\em Proceedings of The American Mathematical Society},
  129(8):2359--2363, 2001.

\bibitem{Clarkson}
J.~A. Clarkson.
\newblock Uniformly convex spaces.
\newblock {\em Trans. Amer. Math. Soc.}, 40:396--414, 1936.

\bibitem{DeMarr}
R.~DeMarr.
\newblock Common fixed points for commuting contraction mappings.
\newblock {\em Pacific J. Math.}, 13(4):1139--1141, 1962.

\bibitem{SyntheticTopology}
M.~Escard\'o.
\newblock Synthetic topology of data types and classical spaces.
\newblock {\em Electronic Notes in Theoretical Computer Science}, 87, 2004.

\bibitem{goebelBGK}
K.~Goebel.
\newblock An elementary proof of the fixed-point theorem of {B}rowder and
  {K}irk.
\newblock {\em The Michigan Mathematical Journal}, 16(4):381--383, 12 1969.

\bibitem{GoebelKirk}
K.~Goebel and W.~Kirk.
\newblock {\em Topics in metric fixed point theory}.
\newblock Cambridge University Press, 1990.

\bibitem{Goehde}
D.~G\"ohde.
\newblock {Z}um {P}rinzip der kontraktiven {A}bbildung.
\newblock {\em {M}ath. {N}achr.}, 30:251--258, 1965.

\bibitem{Halpern}
B.~Halpern.
\newblock Fixed points of nonexpanding maps.
\newblock {\em Bull. Amer. Math. Soc.}, 73(6):957--961, 1967.

\bibitem{Hatcher}
A.~Hatcher.
\newblock {\em Algebraic Topology}.
\newblock Cambridge University Press, 2002.

\bibitem{BrouwerLowerComplexity}
M.~D. Hirsch, C.~H. Papadimitriou, and S.~A. Vavasis.
\newblock Exponential lower bounds for finding {B}rouwer fixed points.
\newblock {\em Journal of Complexity}, 5:379--416, 1989.

\bibitem{Karlovitz}
L.~Karlovitz.
\newblock The construction and application of contractive retractions in
  2-dimensional normed linear spaces.
\newblock {\em Indiana Univ. Math. J.}, 22:473--481, 1972.

\bibitem{KawamuraCook}
A.~Kawamura and S.~Cook.
\newblock Complexity theory for operators in analysis.
\newblock {\em ACM Trans. Comput. Theory}, 4(2):5:1--5:24, May 2012.

\bibitem{KirkFixedPoint}
W.~Kirk.
\newblock A fixed point theorem for mappings which do not increase distance.
\newblock {\em Amer. Math. Monthly}, 72:1004--1006, 1965.

\bibitem{Ko}
K.-I. Ko.
\newblock {\em Complexity Theory of Real Functions}.
\newblock Birkh{\"a}user, 1991.

\bibitem{KohlenbachModuli}
U.~Kohlenbach.
\newblock New effective moduli of uniqueness and uniform a-priori estimates for
  constants of strong unicity by logical analysis of known proofs in best
  approximation theory.
\newblock {\em Numer. Funct. Anal. Optim.}, 14:581--606, 1993.

\bibitem{Kohlenbach}
U.~Kohlenbach.
\newblock On the computational content of the {K}rasnoselski and {I}shikawa
  fixed point theorems.
\newblock In J.~Blanck, V.~Brattka, and P.~Hertling, editors, {\em Proc. of the
  Fourth Workshop on Computability and Complexity in Analysis}, pages 119--145.
  Spinger Lecture Notes in Computer Science LNCS 2064, 2001.

\bibitem{KohlenbachBRS}
U.~Kohlenbach.
\newblock A quantitative version of a theorem due to
  {B}orwein-{R}eich-{S}hafrir.
\newblock {\em Numer. Funct. Anal. Optim.}, 22:641--656, 2001.

\bibitem{KohlenbachMetric}
U.~Kohlenbach.
\newblock Some computational aspects of metric fixed point theory.
\newblock {\em Nonlinear Anal.}, 61:823--837, 2005.

\bibitem{Kohlenbook}
U.~Kohlenbach.
\newblock {\em Applied Proof Theory: Proof Interpretations and their Use in
  Mathematics}.
\newblock Springer, 2008.

\bibitem{KohlenbachLeustean}
U.~Kohlenbach and L.~Leustean.
\newblock Mann iterates of directionally nonexpansive mappings in hyperbolic
  spaces.
\newblock {\em Abstr. Appl. Anal.}, 2003(8):449--477, 2003.

\bibitem{KohlenbachSafarik}
U.~Kohlenbach and P.~Safarik.
\newblock Fluctuations, effective learnability and metastability in analysis.
\newblock {\em Annals of Pure and Applied Logic}, 165(1):266--304, 2014.

\bibitem{Koenigsberger}
K.~K\"onigsberger.
\newblock {\em Analysis 2}.
\newblock Springer, 2004.

\bibitem{KoernleinKohlenbach}
D.~K\"ornlein and U.~Kohlenbach.
\newblock Rate of metastability for {B}ruck's iteration of pseudocontractive
  mappings in {H}ilbert space.
\newblock {\em Numer. Funct. Anal. Optim.}, 35:20--31, 2014.

\bibitem{Krasnoselski}
M.~Krasnoselski.
\newblock Two remarks on the method of successive approximations.
\newblock {\em Uspekhi Mat. Nauk}, 10:123--127, 1955.
\newblock (Russian).

\bibitem{KreinovichPhd}
V.~Kreinovich.
\newblock {\em Categories of space-time models}.
\newblock PhD thesis, Soviet Academy of Sciences, Novosibirsk, 1979.
\newblock (Russian).

\bibitem{KreiselNoCounterexample1}
G.~Kreisel.
\newblock On the interpretation of non-finitist proofs, part {I}.
\newblock {\em J. Symb. Log.}, 16:241--267, 1951.

\bibitem{KreiselNoCounterexample2}
G.~Kreisel.
\newblock On the interpretation of non-finitist proofs, part {II}:
  {I}nterpretation of number theory, applications.
\newblock {\em J. Symb. Log.}, 17:43--58, 1952.

\bibitem{PaulyLeRoux}
S.~Le~Roux and A.~Pauly.
\newblock Closed choice for finite and for convex sets.
\newblock In P.~Bonizzoni, V.~Brattka, and B.~L{\"o}we, editors, {\em The
  Nature of Computation. Logic, Algorithms, Applications}, volume 7921 of {\em
  Lecture Notes in Computer Science}, pages 294--305. Springer Berlin
  Heidelberg, 2013.

\bibitem{ZR}
S.~Le~Roux and M.~Ziegler.
\newblock Singular coverings and non-uniform notions of closed set
  computability.
\newblock {\em Mathematical Logic Quarterly}, 54(5):545--560, 2008.

\bibitem{Leustean}
L.~Leustean.
\newblock A quadratic rate of asymptotic regularity for
  $\operatorname{CAT}(0)$-spaces.
\newblock {\em J. Math. Anal. Appl.}, 325:386--399, 2007.

\bibitem{Megginson}
R.~E. Megginson.
\newblock {\em An introduction to {B}anach space theory}, volume 183.
\newblock Springer, 1998.

\bibitem{Miller}
J.~Miller.
\newblock {\em $\Pi_1^0$-classes in computable analysis and topology}.
\newblock PhD thesis, Cornell University, 2002.

\bibitem{Munkres}
J.~R. Munkres.
\newblock {\em Topology}.
\newblock Prentice Hall, 2000.

\bibitem{thesis}
E.~Neumann.
\newblock {\em {C}omputational {P}roblems in {M}etric {F}ixed {P}oint {T}heory
  and their {W}eihrauch {D}egrees}.
\newblock Master's thesis, Technische Universit{\"a}t Darmstadt, August 2014.

\bibitem{KrasnoselskiWeakConvergence}
Z.~Opial.
\newblock Weak convergence of the sequence of successive approximations for
  nonexpansive mappings.
\newblock {\em Bull. Amer. Math. Soc.}, 73(4):591--597, 1967.

\bibitem{Orevkov}
V.~Orevkov.
\newblock A constructive map of the square onto itself, which moves every
  constructive point.
\newblock {\em Dokl. Akad. Nauk SSSR}, 152:55--58, 1963.

\bibitem{PaulyRepresented}
A.~Pauly.
\newblock A new introduction to the theory of represented spaces.
\newblock {\em arXiv:1204.3763v2}, 2013.

\bibitem{Phelps}
R.~Phelps.
\newblock Convex sets and nearest points.
\newblock {\em Proc. Amer. Math. Soc.}, 8:790--797, 1957.

\bibitem{PourElRichards}
M.~B. Pour-El and J.~I. Richards.
\newblock {\em Computability in Analysis and Physics}.
\newblock Springer, 1989.

\bibitem{ReichStrongConvergence}
S.~Reich.
\newblock Strong convergence theorems for resolvents of accretive operators in
  {B}anach spaces.
\newblock {\em Journal of Mathematical Analysis and Applications},
  75(1):287--292, 1980.

\bibitem{Scarf}
H.~Scarf.
\newblock The approximation of fixed points of a continuous mapping.
\newblock {\em SIAM J. Appl. Math.}, 15(5):1328--1343, 1967.

\bibitem{SchadeKohlenbach}
K.~Schade and U.~Kohlenbach.
\newblock Effective metastability for modified {H}alpern iterations in
  $\operatorname{CAT}(0)$ spaces.
\newblock {\em Fixed Point Theory and Applications}, 2012(1), 2012.

\bibitem{SchroederPhD}
M.~Schr\"oder.
\newblock {\em Admissible {R}epresentations for {C}ontinuous {C}omputations}.
\newblock PhD thesis, Fern{U}niversit\"at {H}agen, 2002.

\bibitem{SchroederAdmissibility}
M.~Schr\"oder.
\newblock Extended admissibility.
\newblock {\em Theoretical Computer Science}, 284:519--538, 2002.

\bibitem{SpeckerSequence}
E.~Specker.
\newblock Nicht konstruktiv beweisbare {S}\"atze der {A}nalysis.
\newblock {\em J. Symbolic Logic}, 14:145--208, 1949.

\bibitem{SpeckerMaximum}
E.~Specker.
\newblock Der {S}atz vom {M}aximum in der rekursiven {A}nalysis.
\newblock In A.~Heyting, editor, {\em Constructivity in Mathematics}, pages
  254--265. Studies in Logic and The Foundations of Mathematics, 1959.

\bibitem{CounterExamplesInTopology}
L.~A. Steen and J.~A. Seebach, Jr.
\newblock {\em Counterexamples in Topology}.
\newblock Springer Verlag, New York, 1978.

\bibitem{TaoMetastability1}
T.~Tao.
\newblock Soft analysis, hard analysis, and the finite convergence principle.
\newblock available online at:
  http://terrytao.wordpress.com/2007/05/23/soft-analysis-hard-analysis-and-the-finite-convergence-principle/,
  May 2007.

\bibitem{TaoMetastability2}
T.~Tao.
\newblock Norm convergence of multiple ergodic averages for commuting
  transformations.
\newblock {\em Ergodic Theory and Dynamical Systems}, 28:657--688, 4 2008.

\bibitem{AbstractStoneDuality}
P.~Taylor.
\newblock A lambda calculus for real analysis.
\newblock {\em Journal of Logic \& Analysis}, 2(5):1--115, 2010.

\bibitem{WeihrauchOriginal}
K.~Weihrauch.
\newblock The degrees of discontinuity of some translators between
  representations of the real numbers.
\newblock Technical Report TR-92-050, International Computer Science Institute,
  Berkeley, 1992.

\bibitem{Weih}
K.~Weihrauch.
\newblock {\em Computable Analysis}.
\newblock Springer, 2000.

\bibitem{Werner}
D.~Werner.
\newblock {\em Funktionalanalysis}.
\newblock Springer, 2011.

\bibitem{Wittmann}
R.~Wittmann.
\newblock Approximation of fixed points of nonexpansive mappings.
\newblock {\em Archiv der Mathematik}, 58(5):486--491, 1992.

\bibitem{ZhangCheng}
Q.-b. Zheng and C.-z. Cheng.
\newblock Strong convergence theorem for a family of {L}ipschitz
  pseudocontractive mappings in a {H}ilbert space.
\newblock {\em Mathematical and Computer Modelling}, 48(3-4):480--485, 2008.

\bibitem{ZieglerDiscreteAdvice}
M.~Ziegler.
\newblock {R}eal {C}omputation with {L}east {D}iscrete {A}dvice: {A}
  {C}omplexity {T}heory of {N}onuniform computability.
\newblock {\em Annals of Pure and Applied Logic}, 163(8):1108 -- 1139, 2012.
\newblock Continuity, Computability, Constructivity: From Logic to Algorithms.

\end{thebibliography}
\nocite{*}
\end{document}